\documentclass[11pt,a4paper]{article}
\usepackage[utf8]{inputenc} 
\usepackage{amsmath}
\usepackage{amssymb}
\usepackage{amsthm}
\usepackage{colortbl,xcolor}

\definecolor{Gray}{gray}{0.85}
\usepackage{epsfig}
\usepackage{verbatim}
\usepackage{color}
\include{pspicture}

\usepackage{hyperref}

\newcommand{\ninseps}[3]{
\begin{figure}[h]
\begin{center}
 \scalebox{#3}{\includegraphics{#1}}
\end{center} 

\vspace{-0.8cm}
\caption{\hspace{0.25cm}#2\label{f:#1}}
\end{figure}
}

\DeclareSymbolFont{AMSb}{U}{msb}{m}{n}
\DeclareMathSymbol{\N}{\mathbin}{AMSb}{"4E}
\DeclareMathSymbol{\Z}{\mathbin}{AMSb}{"5A}
\DeclareMathSymbol{\R}{\mathbin}{AMSb}{"52}
\DeclareMathSymbol{\Q}{\mathbin}{AMSb}{"51}
\DeclareMathSymbol{\I}{\mathbin}{AMSb}{"49}
\DeclareMathSymbol{\C}{\mathbin}{AMSb}{"43}
\DeclareMathSymbol{\F}{\mathbin}{AMSb}{"46}
\DeclareMathSymbol{\LL}{\mathbin}{AMSb}{"4C}
\newcommand{\E}{\mathbb{E}}
\newcommand{\p}{\mathbb{P}}

\newtheorem{corollary}{Corollary}

\newtheorem{theorem}{Theorem}

\newtheorem{proposition}{Proposition}

\newcommand{\var}{\text{var}}

\title{
Analysis of Push-type Epidemic Data Dissemination in Fully Connected Networks
}

\author{Mine \c{C}a\u{g}lar$^{1}$, 
Ali Devin Sezer$^{2,3}$
\\
\\
$^1$ Department of Mathematics\\
 Ko\c{c} University,\\
Istanbul, Turkey\\
\\
$^2$ Laboratoire Analyse et Probabilit\'es\\Universit\'e d'\'Evry Val d'Essonne  \\
91025 \'Evry Cedex, France \\  \vspace{6pt}
\\
$^3$ Institute of Applied Mathematics,\\ Middle East Technical University,\\ Ankara, Turkey}

\begin{document}
\maketitle
\begin{abstract}
Consider a fully connected network of nodes, some of which have
a piece of data to be disseminated to the whole network. We analyze
the following push-type 
epidemic algorithm: in each push round,
every 
node that has the data, i.e., every infected node,  randomly chooses $c \in {\mathbb Z}_+$
other nodes in the network and 
transmits, i.e., pushes, the
data to them.
We write this round as a random walk 
whose each step
corresponds to a random selection of one of the infected nodes; this gives
recursive formulas for the distribution and the moments of the number of newly
infected nodes in a push round. We use the formula for the distribution 
to compute the expected number of rounds so that a given percentage of the network is infected 
and continue a numerical
comparison of the push algorithm and the pull algorithm (where the susceptible nodes randomly choose
peers) initiated in an earlier work. We then derive the
fluid and diffusion limits of the random walk as the network size goes to $\infty$ and
deduce a number of properties of the push algorithm:
1) the number of newly infected nodes in a push round, 
and the number of random selections needed so that 
a given percent of the network is infected, are both asymptotically normal 
2) for large networks, starting with a nonzero proportion of infected nodes, a pull round 
infects slightly more nodes on average
3) the number of rounds until a given proportion $\lambda$
of the network is infected converges to a constant for almost all $\lambda \in (0,1)$. 
Numerical examples for theoretical results are provided.
\end{abstract}


\section{Introduction}
{\em Epidemic algorithms} 
mimic spread of infectious diseases to disseminate data in large networks
 \cite{chierichetti2011rumor, mundinger2008optimal, eugster2004epidemic, karp2000randomized, qiu2004modeling, sanghavi2007gossiping, xu2012push}.
As is common in the literature,
let us call a node of a network {\em infected} if it holds 
the piece of data to be disseminated and {\em susceptible} otherwise.
Two of the main types of epidemic algorithms are {\em push}  and {\em pull}.
Both of these 
progress in discrete stages called {\em rounds}; in a push round,
each infected node randomly
selects $c>0$ nodes uniformly and without repetition among the rest of the nodes and uploads, i.e, pushes, the data
to these nodes; in a pull round, each susceptible node randomly selects $c$ nodes
and if any of these nodes is infected, the selecting node downloads, i.e., pulls, the data from the infected node to which
it has connected.
The 
parameter $c$ is called the {\em fanout.} 

One calls a network fully connected if each of its nodes can directly connect to any other node in the network.
If the network is represented as a graph, the network is fully connected if and only if its graph is
complete.
Analysis over fully connected networks is a natural first step in the study of algorithms on networks.
\cite{karp2000randomized,sanghavi2007gossiping,ganesh2005effect,birman1999bimodal} study epidemic algorithms on
fully connected networks and \cite{gibbens1990dynamic, birman1996computing,coffman2002network} study a range of other stochastic
algorithms
on them; see Section \ref{s:conclusion} for more on fully connectedness and for comments on other topologies. 
The aim of the present work is a thorough analysis of the push algorithm over  
fully connected networks; the following paragraphs explain the elements of this analysis.

A key random variable in epidemic algorithms is the number $Y$ of newly infected nodes after an epidemic round.
The paper
\cite{ozkasap2010analytical}
studies this random variable for fully connected networks
and observes that it is binomial for the 
pull\footnote{\cite{ozkasap2010analytical}, following 
\cite{birman1999bimodal}, reverses the roles of the words ``pull'' and 
``push''; what is called ``pull'' here is called ``push'' in these works.
We always use these words in the sense explained in the first paragraph.
One should keep this reversal in mind when comparing the results of the
present paper with those in \cite{ozkasap2010analytical, birman1999bimodal}.
}
round when conditioned on the number of infected nodes in the network right before the round begins.
For the distribution of $Y$ in the push round,
\cite{ozkasap2010analytical} assumes $c=1$ and derives the formula \eqref{e:transition} by counting all
digraphs which correspond to each realization of $Y$.
The direct computation of \eqref{e:transition}
requires high precision and lengthy arithmetic and this restricts its use to small networks $(n \le 200)$.
In 
Section \ref{s:rwrep} we take a different route and represent $Y$ for general $c$  using
a random walk $S$ with linear and state dependent dynamics.
Recall that within a push round each infected node randomly selects
$c$ peers and transmits its data to these peers.
Each step of the walk $S$ corresponds to one of these random selections; 
see \eqref{e:ilk}, \eqref{e:bern} and \eqref{X1}
for the exact dynamics.
The position $S_k$ of the walk at its $k^{th}$ step 
($k$ being the number of infected nodes before the round begins)
is our desired representation of $Y$.
Subsection \ref{ss:Smodel} explains how $S$ can be used as a model for the whole
push algorithm when it is allowed to take an unlimited number of steps. The next subsection
computes the first and second moments of $S$, which gives, in particular, those of $Y$.

The random walk $S$ is also a discrete time Markov chain and its dynamics therefore can be expressed as its one
step transition matrix $P$.
Thus, one can write (for all $n$, $k$ and $c$) the
distribution of $Y$ as the
first row of $P^k$. 
This gives
a fast algorithm to compute $Y$'s distribution for small $c$,
because $P$ is sparse when $c$ is small.
Section \ref{s:numerical} uses this algorithm to compute
for $n=500$, and $c \in \{1,7\}$,
the expected number of push 
rounds needed so that the proportion of the infected nodes in the network 
reaches $\lambda \in (0,1)$.
This expectation is simple to compute for the pull algorithm, because $Y$ of that algorithm is binomial.
Figures \ref{f:N50}
and \ref{f:DisstimeFanoutNEW} compare the aforementioned expectation for the
push and the pull algorithms.

Section \ref{s:dif} contains the main results of our analysis:  here
we compute  the diffusion and fluid limits of the
random walk $S$ and derive a number of 
properties of the push epidemic algorithm
from these limits.
For the asymptotic analysis to make sense
we set the initial number of infected nodes to $k_n$
such that $\lim_n k_n/n  = (1-\mu) \in (0,1).$
Theorem \ref{t:cthm} shows that as 
$n$ goes to $\infty$, the scaled
random walk $S/n$ behaves like $t\rightarrow \Gamma_t + \frac{1}{\sqrt{n}} X_t$
where $\Gamma$ is the deterministic process $t\rightarrow \mu(1-e^{-ct})$ and $X$ is a time discounted 
Wiener integral of a function of $\Gamma$
(see \eqref{e:formulaX}). 
The $t$ variable here is the continuous scaled time and $t=(1-\mu) = \lim_n k_n/n$ corresponds to
the
$ {k_n} \approx (1-\mu)n^{th}$ step of the random walk and hence to the end 
of the first push round.
This establishes that $\sqrt{n}(Y/n - \Gamma_{1-\mu})$ 
converges to a zero mean normal random variable
whose variance is given by the quadratic variation of $X$ (\eqref{e:varX} and \eqref{e:quadraticvar}). 
Subsection \ref{ss:compforlargen} uses Theorem \ref{t:cthm} 
to compare the pull and the push algorithms for large fully connected networks. 
In particular,
\eqref{e:comparemeans} says
that a pull round always infects slightly more nodes on average, for large
networks and starting with a nonzero proportion of infected nodes.
The difference disappears as $c$ increases.
When the network is initially half infected,
for $c > 15$ a single round of push or pull suffices to infect almost 
all of the nodes.

The last observation suggests that one study more carefully what happens in a single round.
The random walk representation and its limits allow exactly this. 
In subsection \ref{ss:asymptoticstn} we use the asymptotic limit of $S$ derived in Theorem
\ref{t:cthm}
to compute the asymptotics of the number $\tau^{n}_\lambda$ 
of random selections needed so
that the proportion of infected nodes in the network 
reaches $\lambda \in (0,1)$. 
Theorems \ref{t:asymptoticstn}
and \ref{t:asympalpha}
say that this quantity is also asymptotically normal and provide its mean and variance. 
$\tau^{n}_\lambda$ 
is not a function 
of the value of $S$ at a particular deterministic point in time but
of its whole path.
Thus, a stochastic process level analysis of $S$ is inevitable in the study
of $\tau^{n}_\lambda$.

Subsection \ref{ss:fluidlimit} uses $\Gamma$
to derive the fluid limit of a sequence of push rounds. 
Let us denote by $\nu_\lambda^{n}$ the (random) number of push
{\em rounds} needed so that the proportion of the network is above $\lambda \in (0,1).$
The final result of our analysis is Theorem \ref{t:nulambdalimit}, which says that
as the network size increases to $\infty$,
$\nu_\lambda^{n}$
converges
to a constant integer, if $\lambda$ is not one of the
deterministic levels derived in subsection \ref{ss:fluidlimit} that 
the fluid limit of the rounds go through.

Although we have not seen in the prior literature the diffusion
analysis of the random walk $S$, the proof of Theorem \ref{t:cthm}
is based on results and ideas in \cite{ethier1986t} and is 
relegated to the appendix.
We have not been able to find in the prior literature
analyses and proofs similar to the ones
we give in subsections \ref{ss:asymptoticstn} and \ref{ss:verycool} and
therefore the proofs in these subsections follow the statements of the
theorems.

The problems we treat and their solution have connections to a vast literature 
in communication systems, databases, applied probability, queueing theory, 
stochastic biological models among others.
The following review only touches a small subset of this literature
which directly relates to our analysis and of which we happen to be aware of.
\cite{karp2000randomized} studies a number of epidemic
algorithms
in a fully connected network. The one most related to the current
paper is an algorithm in which {\em all} nodes randomly connect to peers
rather than only
infected or only susceptible and all connections use both pull and push.
The paper uses Chernoff's bound to derive bounds 
on the tail probabilities on
the number of rounds this algorithm needs to spread a piece of data to 
the whole network with
high probability.
\cite{sanghavi2007gossiping} studies the effect of dividing the
data to be transferred into pieces on the performance of the
epidemic algorithms in a fully connected network
and derives asymptotic 
bounds on the tail probabilities of the 
number of rounds to disseminate the data to the whole network. 
As with \cite{karp2000randomized}, 
in \cite{sanghavi2007gossiping} the rounds are considered atomic
and during each round all nodes randomly connect to peers.
The main mathematical tool is again large
deviation bounds on independent and identically distributed (iid) sums, similar to Chernoff's bound. 
\cite{ganesh2005effect} studies a continuous time Markov process model similar to actual epidemic models
and uses a coupling argument to find bounds on the expected time to total dissemination in terms of
the largest eigenvalue of the adjacency matrix of the network graph and applies its results to a number
of graph structures, including complete graphs.
\cite{xu2012push} studies epidemic algorithms as a model for the spread of computer viruses; in this context
it makes sense to allow nodes to ``recover.''
Then a natural
quantity of interest is the time limit of the probability of each node being infected or susceptible.
\cite{xu2012push} proves the existence of these limits and derives
conditions 
under which convergence occurs exponentially fast.
\cite{akdere2006comparison}
compares pull, push and and a ``pull-push'' algorithm in the context of sensor networks using 
software that is used in actual sensors and a simulation environment which can run this software. This allows
its authors to study several aspects
of the performance of these algorithms in practical systems.

To the best of our knowledge the present paper is the first to use diffusion limits in the study
of epidemic algorithms in networks. However, in the biologic epidemics literature diffusion limits
are a basic tool; the classical reference on this subject is \cite[Chapter 11]{ethier1986t}; a recent
review is \cite{britton2010stochastic}. In all biological epidemic models
that we are aware of, the network graph is implicitly taken to be fully connected
by assuming that all members of the population 
somehow are able to interact with each other similar to chemicals
interacting in a liquid mixture.
There are two classes of epidemic models: continuous time and
discrete time \cite{daley2001epidemic}.
The continuous time assumption of the first class 
leads to a limit process 
(see \cite[Theorem 2.3, page 458]{ethier1986t})
which is different from the asymptotics of the push algorithm.
Among the works which assume a discrete time, the analysis
of \cite{tuckwell2007some} is closest to this work. The authors
of \cite{tuckwell2007some}
study a discrete time epidemic process which allows recovery. Infection
mechanism of the model corresponds to the pull algorithm (susceptibles
randomly choose peers). Besides allowing recovery the novelty of
\cite{tuckwell2007some} is that it allows a random fanout for each individual.
\cite[Section 3.2]{tuckwell2007some} derives a diffusion limit for this
model for constant fanout and no recovery.

Another branch of research related to epidemic algorithms is urn models in applied probability. 
The paper \cite{chierichetti2011rumor} uses this connection in finding asymptotic bounds on
the tail distribution of the the number of rounds until most nodes are infected, for the pull, push and a ``push-pull''
algorithms over graphs defined by the classical preferential attachment model \cite{bollobas2001degree}.
After our analysis we have noticed
that \cite{gani2004random} treats an urn model that corresponds to the push algorithm over fully connected networks.
In particular, 
it derives the asymptotic limit of the random variable $Y$
using a 
different set of tools from the ones used in this work  
(namely, probability generating functions (pgfs) and a result of \cite{warren1996peaks} which
characterizes pgfs arising from a Bernoulli sequence).  \cite{gani2004random} notes that a recursive characterization
of $Y$, similar to the one we give in Section \ref{s:rwrep} goes all the way back to \cite{woodbury1949probability}, which casts the problem in terms
of a sequence of trials of an event ``presuming that the probability
of the event on a given trial depends only on the number of the
previous successes.''

Further comments on our results and future research 
are in Section \ref{s:conclusion}.

\section{Random Walk Representation}\label{s:rwrep}
We will begin our analysis by considering a single push round on a network
with $n$ total nodes and $k$ infected nodes.
Let $Y(n,k)$ denote the number of the number of newly infected nodes after
the round is over.
The paper \cite{ozkasap2010analytical}, which only considers the case $c=1$, finds
the following formula for the distribution of
$Y(n,k)$
\begin{equation}
	\p(Y(n,k)=i)=
  \frac{ \binom{n-k}{i} i! \sum_{k_1 =i}^{k} \binom{k}{k_1}(k-1)^{k-k_1}
  {\bf s}(k_1,i)}{(n-1)^{k}},                        \label{e:transition}
\end{equation}
where ${\bf s}$ denotes the Stirling numbers of the second kind.
This formula comes from representing the result of a round as a graph and counting
those graphs that give $Y(n,k) = i.$
The use of \eqref{e:transition} raises two issues: 1) it
is valid only for $c=1$ and 2) the expressions that appear in it can be computed exactly
for only small values of $n$ (i.e., $n\le 200$) ,
see \cite{ozkasap2010analytical} for more on these.
In this section we derive a new dynamic representation of this round
as a random walk with state dependent increments, whose each step corresponds to a
random selection made by one of the infected nodes during the round.

Let us first consider the case $c=1$; the extension to $c>1$ will be straightforward.
The random walk representation of the push round begins with thinking that the nodes do their
random  selection of a peer from the rest of the network
one by one. 
The Bernoulli random variable
$X_1$ denotes the result of
the first selection  ($X_1=1$ if the selected node is susceptible, $0$ otherwise),
$X_2$ the result of the second selection,
and so on.
Now define
\begin{equation}  \label{e:ilk}
S_{l+1} = S_l + X_{l+1}         \qquad l=0,1,2,\ldots
\end{equation}
where the conditional distribution of $X_{l+1}$ given $S_l=i  \le n-k$ is
Bernoulli with success probability
\begin{equation} \label{e:bern}
\p(X_{l+1} = 1|\ S_l=i) = (n-k-i)/(n-1), 
\end{equation}
for $l=1,\ldots, k-1$ and $S_0=0$. 
$S_l$ is the number of newly infected nodes after 
the $l^{th}$ random selection;
$Y(n,k)$, then, is $S_{k}.$
The conditional distribution \eqref{e:bern} is written by noting that
the only way for the $l+1^{st}$ selection to increase the infected
count is by choosing a susceptible node that has not been touched
by the first $l$ random selections. We also note
$\mathbb{P}( X_{l+1} = i | S_0, S_1,...,S_l ) = \mathbb{P}(X_{l+1}= i | S_l)$ which
makes $S$ a Markov process and ensures that \eqref{e:bern} 
determines the entire distribution of $S$.

Subsection \ref{ss:Smodel} will explain
how $S$
can serve as a model of the whole push algorithm 
when it runs for an unlimited number of steps.

The dynamics \eqref{e:bern} imply that $S$ is
a Markov chain with one step transition matrix
\begin{equation}\label{e:P}
P_{i,i+j} =  {\mathbb P}( X_{l+1} = j | S_l = i).
\end{equation}
The probability distribution of $Y(n,k)$ (i.e., the position of the chain
at its $k^{th}$ step)
is the first row of $P^k$, i.e.,
 \begin{equation}  \label{e:katki}
\p(Y(n,k)=i)=P^k_{0,i}\qquad i=0,1,2,\ldots,
\min(k,n-k).
 \end{equation}
$P^k$ can be computed quickly for relatively large values
of $n$, because $P$ is sparse.

For $c>1$,
one only
generalizes the conditional distribution of $X_{l+1}$ given $S_l=i \le n-k$
from \eqref{e:bern} to
\begin{equation}\label{X1}
\p(X_{l+1} = j |\ S_l=i) =
\binom{k-1+i}{c-j}\binom{n-k-i}{j}\Big/ \binom{n-1}{c},
\end{equation}
which is a hypergeometric
distribution 
on 
$\{j: \max(0,c-(k+i-1)) \le j \le \min(c,n-(k+i))\}$.
With this generalization, the formulas \eqref{e:P} and \eqref{e:katki} 
continue to work for $c>1$.

\eqref{X1}, \eqref{e:P} and \eqref{e:katki} imply that $Y$ can
take values from $\max(0,c+1-k)$ to $\min(k,n-k)$ with positive probability.


\subsection{$S$ as a model of the whole push algorithm}
\label{ss:Smodel}
Now suppose that we would like to model a second round which follows the
first round also with a random walk. Let us temporarily call this
random walk $S^{(2)}$. 
\eqref{e:ilk} and \eqref{e:bern} continue to describe the dynamics
of $S^{(2)}$ if we only replace the $k$ in \eqref{e:bern} with $S_k + k$,
which is the number of infected nodes in the network
at the end of the first round. This and the Markov property of $S$
imply that we don't actually need
a second process $S^{(2)}$ to describe the second round and it suffices to
simply run the original process $S$ indefinitely; its first $I_0= k$
steps will model the random selections of the infected nodes in the
first round, its next
 $I_1 \doteq k+ S_k$ steps will model the random selections in 
the second round, the next $I_2 \doteq k + S_{I_0 + I_1} $ 
steps in the third round,
the next $I_3 \doteq k + S_{I_0+ I_1+ I_2} $ steps in the fourth round
and so on. 
The sequence $(I_0,I_1,I_2,..)$ itself represents the number of infected
nodes after the $0^{th}$ round, the $1^{st}$ round, the $2^{nd}$ round and so on.
Thus we see that the single random walk $S$, when
ran indefinitely, is another model for the entire push algorithm.
The key difference from the traditional model of a sequence
of push rounds is that $S$ takes the
random selections that occur in the rounds
as the atomic operation of the push algorithm and not
the rounds; the rounds are then expressed as recursive random increments of this walk as above.
We will use this observation repeatedly in subsections 
\ref{ss:asymptoticstn} and \ref{ss:verycool} when we want to use 
results about $S$ to get results on a sequence of push rounds and
hence on the entire push algorithm. In the rest of the paper
we will always assume that $S_l$ is defined for all $l$, following
the dynamics \eqref{e:ilk} and \eqref{e:bern} or \eqref{X1} (for $c\ge 1).$

\subsection{Expectation and second moment of $S$ and $Y$}
By taking the expectation of both sides of \eqref{e:ilk}
and of its square
and using \eqref{e:bern}
one can find linear
recursions for ${\mathbb E}[S_l]$ and ${\mathbb E}[S_l^2]$;
setting $l=k$ gives ${\mathbb E}[Y]$ and ${\mathbb E}[Y^2]$.
The results of this subsection will also be useful in the
fluid and diffusion limit analysis of Section \ref{s:dif}.
Let us start with fanout $c=1$.
Take the conditional expectation given $S_l$ of both sides 
of \eqref{e:ilk} to get
$\E [S_{l+1} | S_l ]=\E [S_l| S_l] + \E[X_{l+1}|S_l].$
The conditional distribution
\eqref{e:bern} of $X_{l+1}$ given $S_l$
implies that ${\mathbb E}[X_{l+1}|S_l] = (n-k-S_l)/(n-1)$.
Substituting this in the last display and taking now the ordinary
expectation of both sides give
\begin{equation} \label{e:ilkrec}
\gamma_{l+1} = a_1 \gamma_l + a_0
\end{equation}
where  $\gamma_l \doteq {\mathbb E}[S_{l}]$,
and
\begin{equation}\label{e:coefa}
 a_1 \doteq (n-2)/(n-1),~~~ a_0\doteq  (n-k)/(n-1).
\end{equation}
By definition we set $\gamma_0 \doteq 0.$ 
\eqref{e:ilkrec} is a linear recursion and its solution
is
\begin{equation}\label{e:gammal}
\gamma_l =  (n-k)
(1- a_1^l ).
\end{equation}
$Y$, the number of infected nodes at the end of the push round,
equals $S_k$ and therefore
\begin{equation}  \label{exp1}
{\mathbb E}[Y]= \E\, [S_k] =(n-k)(1- a_1^k).
\end{equation}

The second moment of $S_l$ is computed similarly. Let
$\alpha_l \doteq {\mathbb E}[S_l^2]$;
$\alpha_l = {\mathbb E}[ S_l^2 ] =
{\mathbb E}\left[ (S_{l-1} + X_l)^2\right]
= {\mathbb E}[ S_{l-1}^2] + 2{\mathbb E}[ S_{l-1}X_l] + {\mathbb
E}[X_l^2].$
The middle term can be written in terms of $\alpha_{l-1} $ and $\gamma_{l-1}$
as follows:
\begin{align*}
{\mathbb E}[S_{l-1} X_l] &={\mathbb E}[{\mathbb E}[S_{l-1} X_l |
S_{l-1}] ]
	 ={\mathbb E}[S_{l-1}{\mathbb E}[ X_l | S_{l-1}] ]\\
	 &={\mathbb E}\left[S_{l-1}\frac{n-k - S_{l-1}}{n-1}\right]= \frac{n-k}{n-1} \gamma_{l-1} - \frac{\alpha_{l-1}}{n-1},
	 \end{align*}
where we have again used the conditional distribution \eqref{e:bern}.
	 Then, we have the following recursion for $\alpha_l$:
	 \begin{align}
	\label{e:alprec}
 \alpha_l &= \alpha_{l-1} +2\left( \frac{n-k}{n-1}\, \gamma_{l-1} -
 \frac{\alpha_{l-1}}{n-1} \right) + \frac{n-k-\gamma_{l-1}}{n-1} \notag \\
	 &= b_2 \alpha_{l-1}
	 + b_1 \gamma_{l-1}
	 + b_0
	\end{align}
where 
\begin{equation}\label{e:coefb}
b_2 \doteq \frac{n-3}{n-1}, b_1 \doteq \frac{2(n-k) - 1}{n-1},
b_0 \doteq \frac{n-k}{n-1}.
\end{equation}
This and \eqref{e:gammal}
imply
\begin{equation}  \label{e:alphal}
\alpha_l = \left(1- b_2^l\right)b_0 + 
b_1 \sum_{i=0}^{l-1} \gamma_{i} b_2^{l-1-i}.
\end{equation}
The second moment of $Y$ is 
\begin{equation}\label{e:secmomY}
{\mathbb E}[Y^2] = {\mathbb E}[S_k^2] =\alpha_k.
\end{equation}
\eqref{e:gammal} implies
$\lim_l \gamma_l = (n-k)$. Hence, the last sum in \eqref{e:alphal} converges
to $(n-k) \sum_{i=0}^\infty b_2^i = (n-k)(n-1)/2$ and this implies
$\alpha_l \rightarrow (n-k)^2$.
Therefore, for $n$ and $k$ fixed, 
$\var(S_l) \rightarrow 0$ as $l\rightarrow \infty$
and $S_l \rightarrow (n-k)$, i.e., if
the random selections of the nodes continue indefinitely all nodes will eventually be infected almost surely.
The graph of the variance $\var(S_l)$ is shown in
Figure \ref{f:variance}. \ninseps{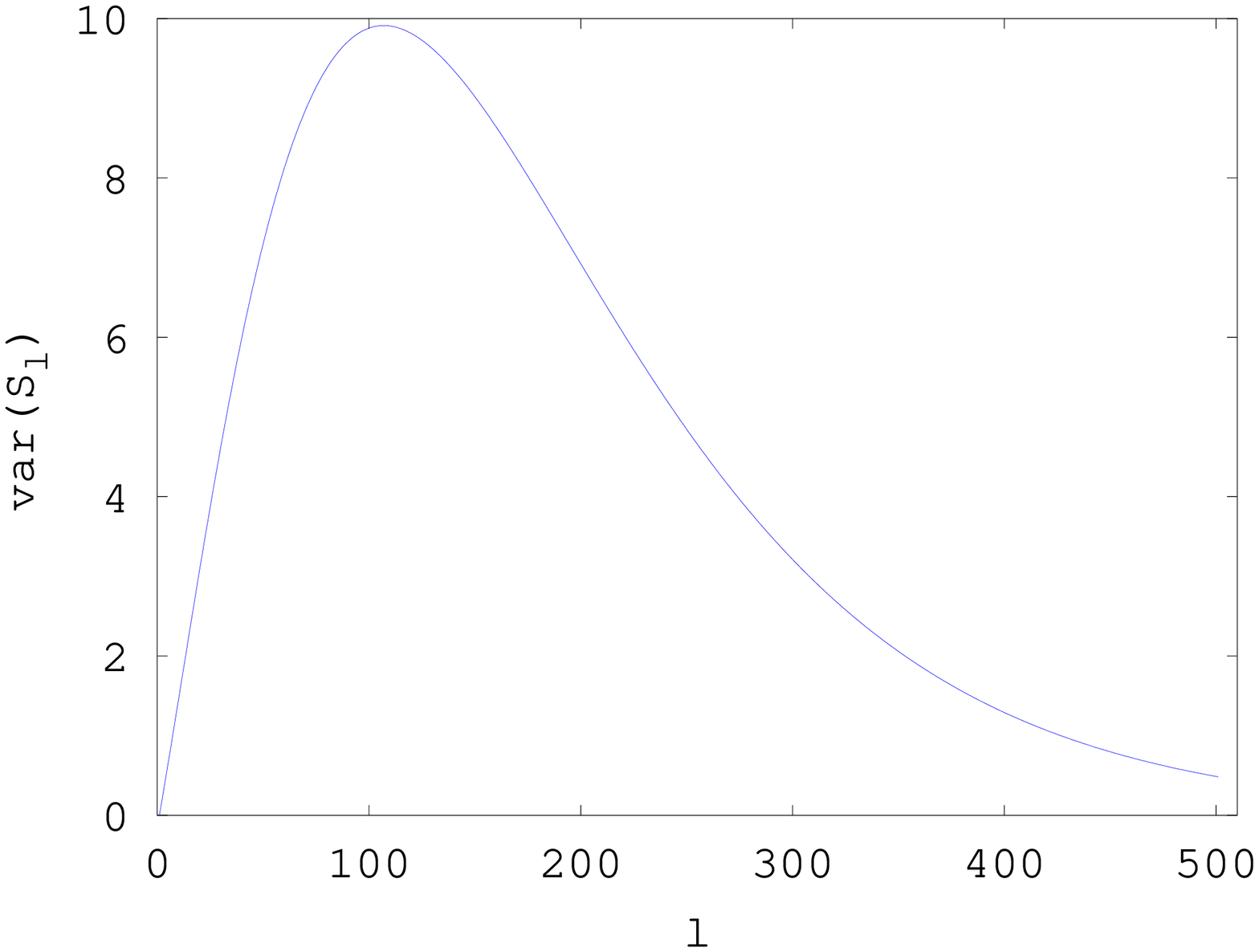}{Graph of $\var(S_l)$ for
$n=100, k=20$}{0.4} 

The case $c>1$ works the same way except that one uses
the conditional distribution \eqref{X1} 
rather than \eqref{e:bern} in computing ${\mathbb E}[X_{l+1}^q|S_l]$, $q=1,2$; 
all of the equations \eqref{e:ilkrec}, \eqref{e:gammal}, \eqref{e:alprec}
and \eqref{e:alphal} remain as before except that one
generalizes the definitions of the coefficients $\{a_i\}$ \eqref{e:coefa}
and $\{b_i\}$ \eqref{e:coefb} to 
\begin{align}\label{e:coefc}
 a_0 &\doteq c(n-k)/(n-1),~~~ a_1 \doteq (n-1-c)/(n-1),\\
b_0 &\doteq \frac{c((n - (c+1)) + (n-k)(c- 1))(n-k)}{(n-2)(n-1)},\notag\\
b_1 &\doteq \frac{c  ( n - (c+1) )(2(n-k) -1 )}{(n-2)(n-1)},~~~
b_2 \doteq 1 - \frac{2c}{n-1} + \frac{c(c-1)}{(n-1)(n-2)}.\notag
\end{align}
Note that these reduce to \eqref{e:coefa}  and \eqref{e:coefb} for
$c=1.$
We summarize the formulas derived about the distribution of $Y(n,k)$ in Table \ref{table1}.
\begin{table}
\begin{tabular}{ !{\color{gray}\vline}c !{\color{gray}\vline} c !{\color{gray}\vline} c !{\color{gray}\vline} c !{\color{gray}\vline}}
\arrayrulecolor{gray}
\hline 
     &  Distribution &  Expectation &   Second moment   \\ 
\arrayrulecolor{gray}
\hline
 $c=1$ & \eqref{e:P},  \eqref{e:katki}  &  \eqref{exp1} & \eqref{e:secmomY}
\\ 
\arrayrulecolor{gray}
 \hline
change for 
$c\geq 1$ 
&\eqref{e:bern}$\rightarrow$ \eqref{X1} &
\eqref{e:coefa} $\rightarrow$ \eqref{e:coefc} &  \eqref{e:coefb} $\rightarrow$
\eqref{e:coefc} \\
\arrayrulecolor{gray}
\hline 
 \end{tabular}
  \centering 
\caption{Equation numbers for the distribution and moments of $Y(n,k)$}
  \label{table1}
\end{table}

\section{A numerical comparison of push and pull}\label{s:numerical}
Let us now use the results so far to numerically
compare the push and the pull algorithms. The observations made in this section
will also motivate the theoretical results of the next chapter.
As in subsection \ref{ss:Smodel}, let
$I_m$ denote the number of infected
nodes after the $m^{th}$ round. We have shown in that subsection how to write
$I_m$ in terms of $S$. The Markov property of $S$ implies that
one can also write the same sequence as
\begin{equation}\label{e:dynamicsI}
I_m = I_{m-1} + Y_m,
\end{equation}
where $Y_m$, conditioned on $I_{m-1}=j$,
is independent of $(I_0,I_1,...,I_{m-2})$
and has the same distribution as $Y(n,j)$.
One of the key questions about the process $(I_0,I_1,I_2,...)$,
and hence about the push algorithm
is this:
how many rounds is required so that
a given proportion of the network is infected? This is the quantity
that all of \cite{ozkasap2010analytical, karp2000randomized, sanghavi2007gossiping, chierichetti2011rumor} analyze. The answer to this question
is expressed as the following stopping time of $I$:
\begin{equation}\label{e:nunlambda}
\nu_\lambda^{n} \doteq \inf \{ m: I_m/n \ge \lambda \}, \lambda \in (0,1).
\end{equation}
We compute the weak limit of 
$\nu^{n}_\lambda$ as the network size goes to $\infty$ 
in subsection \ref{ss:verycool}.
In the numerical study of the present section, 
we compute
$N(\lambda,k) \doteq {\mathbb E}_k\left[\nu_\lambda^{n}\right]$,
for $n = 500$ and $c \in \{1,7\}$;
the $k$
in the subscript of the expectation operator denotes that we condition 
on $I_0 = k$, i.e.,
the infected number of nodes in the network before the first round is $k$. 
$N(1,k)$ is called the mean total dissemination time. 
Because $n$ is fixed (i.e., we are not taking any limits)
throughout this section,
there is no harm in assuming $\lambda \in
\{ j/n, 0 < j < n \}$, which is the set of all
possible proportions of
infected nodes for a finite network with $n$ nodes.
The dynamics
\eqref{e:dynamicsI} implies that for $\lambda >k/n$
\begin{equation}\label{e:recN}
N(\lambda ,k) = 1 +  N(\lambda ,k) \p(Y(n,k) = 0)
        +   \sum_{i =1}^{n \lambda -k-1}  N(\lambda ,k+i ) \p ( Y(n,k) = i).
\end{equation}
The $1$ on the right means that going from $k$ infected nodes
to $j = n\lambda$ infected nodes will take at least one round, the second term handles the case where
no infections occur in the first round and the sum handles the cases where
the first round infects at least $1$ node but less than the $j-k$ needed 
to get a total of $j$ infected nodes.
Furthermore
\begin{equation}\label{e:N0}
N(\lambda,k) = 0 
\end{equation}
when $\lambda \le k/n$ because if there are $k$ infected nodes initially
then obviously $j = \lambda n  \le k$ of them 
are also infected already and we need no rounds.
\eqref{e:N0}, \eqref{e:recN} and \eqref{e:katki} can be used to compute 
$N(j/n,k)$ for all $j$ and $k$.
Note that \eqref{e:recN} and \eqref{e:N0} are the same for the pull algorithm as well, the only change is in the
distribution of $Y(n,k)$, which is binomial for the pull. 
Figure \ref{f:N50} shows $\lambda \rightarrow N(\lambda,k)$ for the push and
the pull algorithms for $k=50$, $n=500$ and $c\in \{1,7\}$. 
These values of $k$ and $n$ correspond to $k/n = 50/500 = 0.1$ initial proportion of infected nodes.
\ninseps{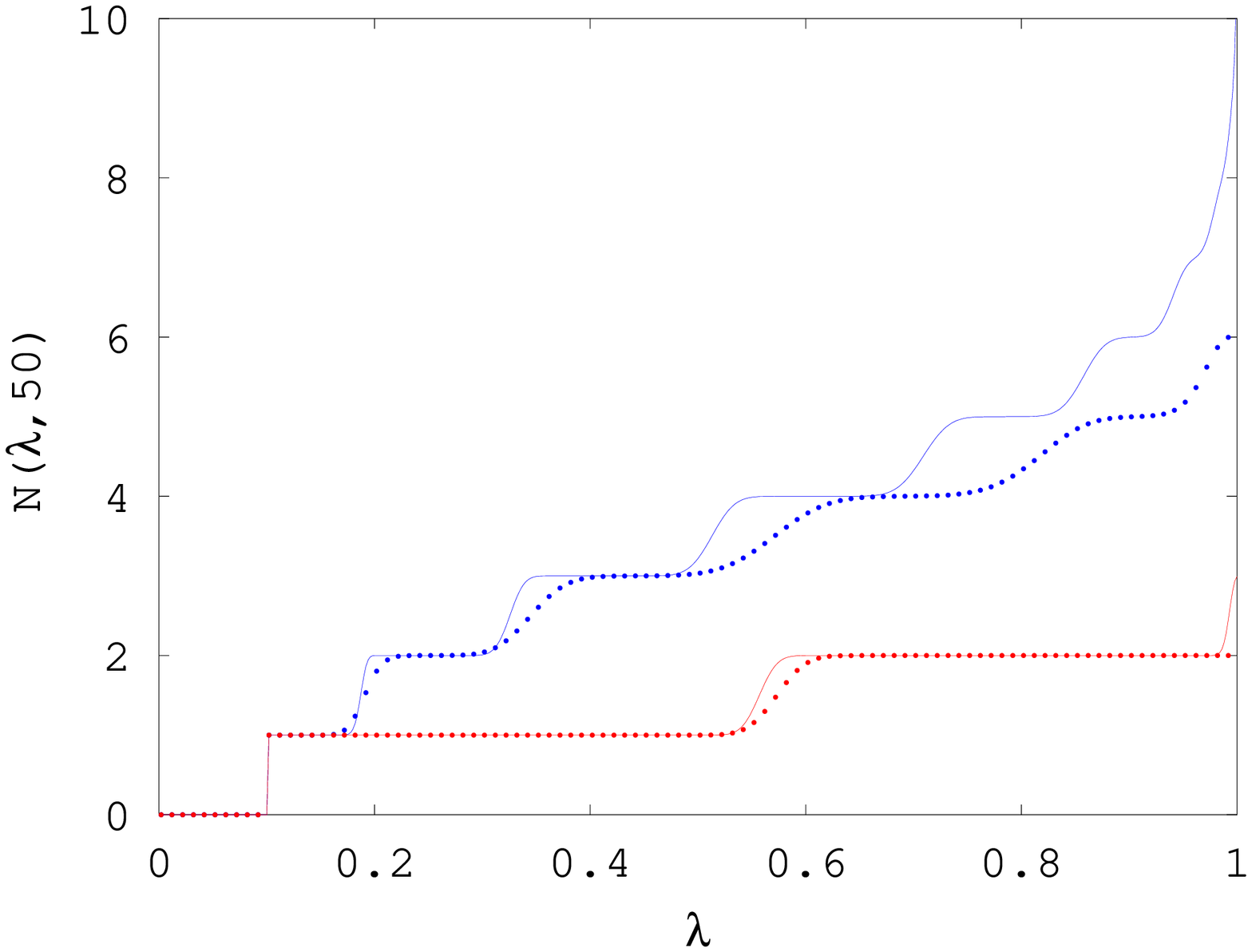}{Graphs of $\lambda \rightarrow N(\lambda,k)$ 
for the push and the pull algorithms; the dotted curves
are pull, the solid curves are push; $n=500$, $k=50$, 
$c\in \{1,7\}$, the top couple is $c=1$ and the bottom couple is $c=7$}{0.45}

Figure \ref{f:N50} suggests that, for both push and pull,
$\lambda \rightarrow N(\lambda,k)$ alternates
between phases of constancy and rapid growth and that this behavior gets more 
marked as $c$ increases. One of the goals of the next section is to explain this
behavior. For now, let us briefly comment that
both of these algorithms have deterministic fluid limits (derived for the push algorithm
in Theorem \ref{t:cthm} and in subsection \ref{ss:fluidlimit}) and as the network size grows
each round infects an almost deterministic proportion of the nodes.
This can be used to prove that, in the limit, $\nu^{n}_\lambda$ becomes almost deterministic and
as a function of $\lambda$ it becomes a step function, increasing only at the levels of infection
attained by the rounds of the fluid limit; the exact result on this is
Theorem \ref{t:nulambdalimit}, proved in subsection \ref{ss:verycool}. 
The more pronounced nature of the growth phases for greater values of $c$ will again be explained by Theorem \ref{t:cthm}
which implies that the deviations from the fluid limit has a lower variance as $c$ grows.

The $N$ of the pull algorithm in Figure \ref{f:N50}
lies on or below that of push. This suggests, for a large network
with a nonzero initial proportion of infected nodes, on average,
the pull reaches a given level in less or equal number of rounds than 
the push. 
However, note that the difference between the algorithms in Figure 
\ref{f:N50} 
is not that great and grows only as the network nears complete infection.
The theoretical result which
explains these observations is Proposition \ref{t:comppullpushmean} in subsection \ref{ss:compforlargen}, which compares
the expected number of infected nodes in a pull and a push round.


\cite[Section 6]{ozkasap2010analytical} uses the binomial distribution of $Y$ under the pull 
algorithm
to compute the dependence 
on the fanout $c$
of 
the mean total
dissemination time $N(1,1)$ 
starting with a single infected node for a network of $500$ nodes.
With \eqref{e:katki} we are able
to do the same also for the push algorithm. 
$N(1,1)$
as a function of $c$ for $n=500$ is given in Figure \ref{f:DisstimeFanoutNEW} for both algorithms.
Although $N(1,1)$ is larger
in the push algorithm for smaller fanout values, this value for 
both algorithm seem to converge for $c\geq 8$. Corollary \ref{c:compare} below partially explains
this phenomenon.
 

\ninseps{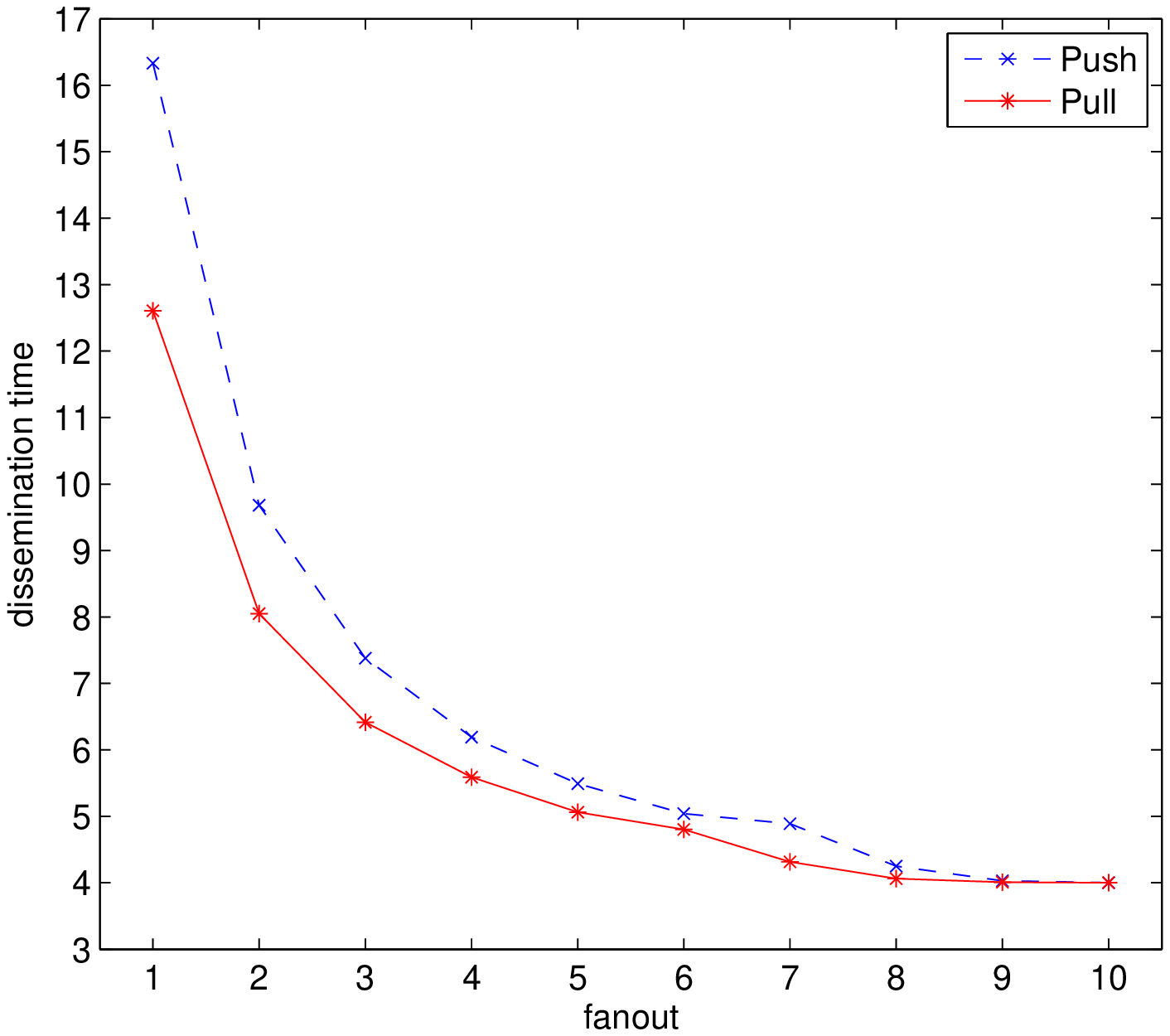}
{Expected total dissemination time in rounds versus fanout for $n=500$, and $k=1$}
{0.5}

\section{Fluid and Diffusion Limits}\label{s:dif}
A great deal can be understood about the push algorithm by
computing the fluid and diffusion limits of $S$
as the network size $n$ goes to $\infty$.
To get meaningful limits, we allow the initial number of infected nodes $k$ 
to depend on $n$ in such a
way that
\begin{equation}\label{e:defmu}
0 < \mu = \lim_{n \rightarrow \infty} \frac{n-k_n}{n}< 1
\end{equation}
holds.
The limits we talk of here are known as 
{\em weak limits} in probability theory and is almost always shown
with the sign $\Rightarrow$, which we will also do below.
The quintessential weak convergence
result is the central limit theorem and ``diffusion approximations''
are central limit theorems for the entire sample paths of processes.
Two of the basic references on weak convergence are 
\cite{BillWeakConv,ethier1986t}.

To get the fluid and diffusion limits, we scale and center $S$
and time as follows:
\begin{equation}
X^n_t \doteq \label{scaled}
\sqrt{n } \left( \frac{S_{\lfloor nt\rfloor }}{n} - \Gamma^n_t \right),
~~~ \Gamma^n_t \doteq \frac{\gamma_{\lfloor nt \rfloor}}{n}.
 \end{equation}
The time variable $t$ of the scaled processes
$\Gamma^n$ and $X^n$ correspond to the $\lfloor n t\rfloor^{th}$ step of
$S$; $\Gamma^n$ is the expected proportion of newly infected nodes at
time $t$ and $X^n$ is $\sqrt{n}$ times the deviation of the actual 
proportion from the expected proportion, again at time $t$. 
This is the standard scaling in all diffusion analyses of Markovian
random walks with finite variance increments.
Just as in the central limit theorem,
the scaling by $\sqrt{n}$ puts the difference between the actual
and the expected proportions at a scale that ensures weak convergence
to a nontrivial (i.e., neither $0$ nor $\infty$) limit.

The process $(X^n,\Gamma^n)$ takes values in $D_{{\mathbb
R}^2}[0,\infty)$,
 the vector space of right continuous functions with left limits from
 $[0,\infty)$ to ${\mathbb R}^2$. 
Define
\begin{equation} \label{sigma}
 \Gamma_t \doteq  (\Gamma_0 - \mu)   e^{-ct} + \mu,~~
\sigma_t \doteq \sqrt{c (\mu - \Gamma_t)(1-(\mu-\Gamma_t))}.
\end{equation}
The next theorem gives the fluid and diffusion limits of $S$.
\begin{theorem} \label{t:cthm} Let $(X^n,\Gamma^n)$ 
be defined as in
\eqref{scaled} and $\Gamma$ as in \eqref{sigma}. Let $X$ be given by
\begin{align}\label{e:formulaX}
X_t &\doteq  X_0 e^{-ct}  + e^{-ct} \int_0^t e^{cs}
\sigma_s dW_s, 
\end{align}
where $c\in \mathbb{Z}_+$ and $\sigma$ is as in \eqref{sigma}.
Then
$(X^n,\Gamma^n) \Rightarrow (X,\Gamma).$
\end{theorem}
The proof, given in the appendix,
is based on representing Markov processes and their
weak convergence in terms of
the semigroups that the processes define and the generators of these semigroups
\cite{ethier1986t}.

Theorem \ref{t:cthm} implies that the random walk 
$\{S_l\}$
behaves more like\\
$\left\{ n\left( \Gamma_{l/n} + \frac{1}{\sqrt{n}} X_{l/n}\right)\right\}$
as $n$, the network size, increases.
The left part of
Figure \ref{f:samplepaths} shows $\{n\Gamma_{l/n}\}$ and a sample path
of $\{S_l\}$ and its right part
shows their difference. The random path of $S_l$ in this figure has
been simulated using \eqref{e:bern}.
Theorem \ref{t:cthm}
implies that 
the pathwise distribution of this difference gets closer 
to that of $\{\sqrt{n} X_{l/n}\}$  as $n$ increases.

Remember that $Y(n,k_n)$ can be represented as
the value $S_{k_n}$ of $S$ at step $k_n$, which corresponds to time $1-\mu$ in
the scaled continuous time. Theorem \ref{t:cthm} then implies
in particular that 
$n \Gamma_{1-\mu} + \sqrt{n} X_{1-\mu}$ is the normal approximation
of $Y(n,k_n)$. Let us write this as
\begin{equation}\label{e:approxY}
Y(n,k_n) \approx
n \Gamma_{1-\mu} + \sqrt{n} X_{1-\mu}.
\end{equation}
The random variable on the right is normally
distributed with mean $n \Gamma_{1-\mu}$. To compute
its variance we only need the second moment of $X_{1-\mu}$, which we derive
now.  It will be simpler to write everything in terms of
\begin{equation}\label{e:defXbar}
\bar{X}_t \doteq  \int_0^t e^{cs}  \sigma_s dW_s.
\end{equation}
The second moment of $X_{t}$ is 
\begin{equation}\label{e:varX}
\var(X_t) \doteq e^{-2ct} {\mathbb E}[\bar{X}_t^2]. 
\end{equation}
 $\bar{X}_t$ is a stochastic integral with respect to
a Brownian motion, and therefore $t\rightarrow \bar{X}_t$ is a continuous martingale whose quadratic
variation equals \cite[page 139]{karatzas1991brownian}
\begin{align}\label{e:quadraticvar}
\langle \bar{X}\rangle_t &= \int_0^t e^{2cs} 
\sigma_s^2 ds 
	 = \mu \left(e^{ct}-1 - \mu ct \right ),
\end{align}
which is a deterministic process. This implies
${\mathbb E}[\bar{X}_t^2] = \langle \bar{X} \rangle_t$ 
(see again \cite[page 137]{karatzas1991brownian}).
This and setting $t=(1-\mu)$ in \eqref{e:varX} gives
\begin{equation}\label{e:varY}
n \mu e^{-2c(1-\mu)} \left(e^{c(1-\mu)}-1 - \mu(1-\mu)c \right)
\end{equation}
as our approximation of the variance of $Y(n,k_n).$

\begin{figure}[h]
\begin{center}
\scalebox{0.3}{
\centerline{\epsfig{file=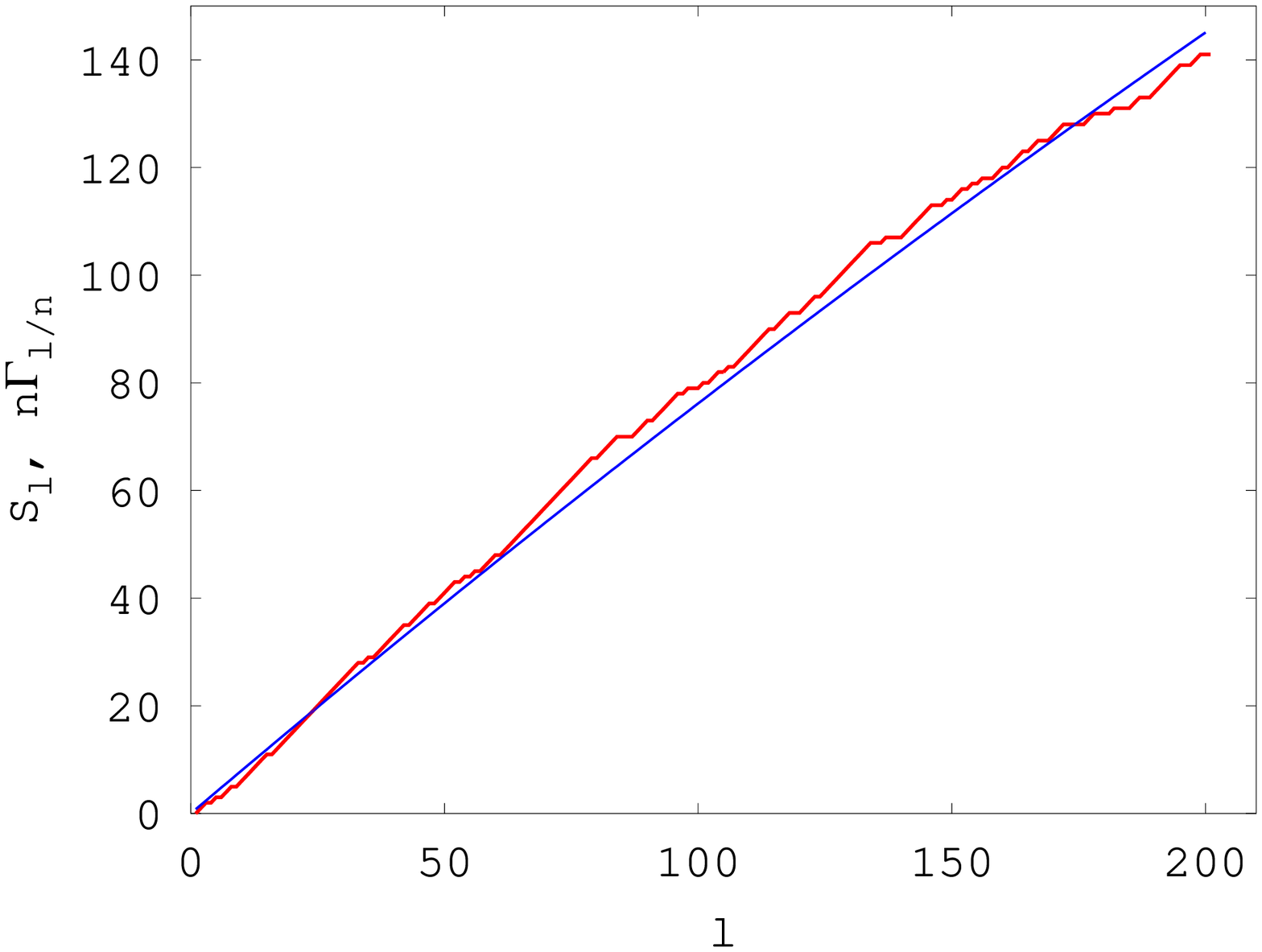}\hspace{2cm} \epsfig{file=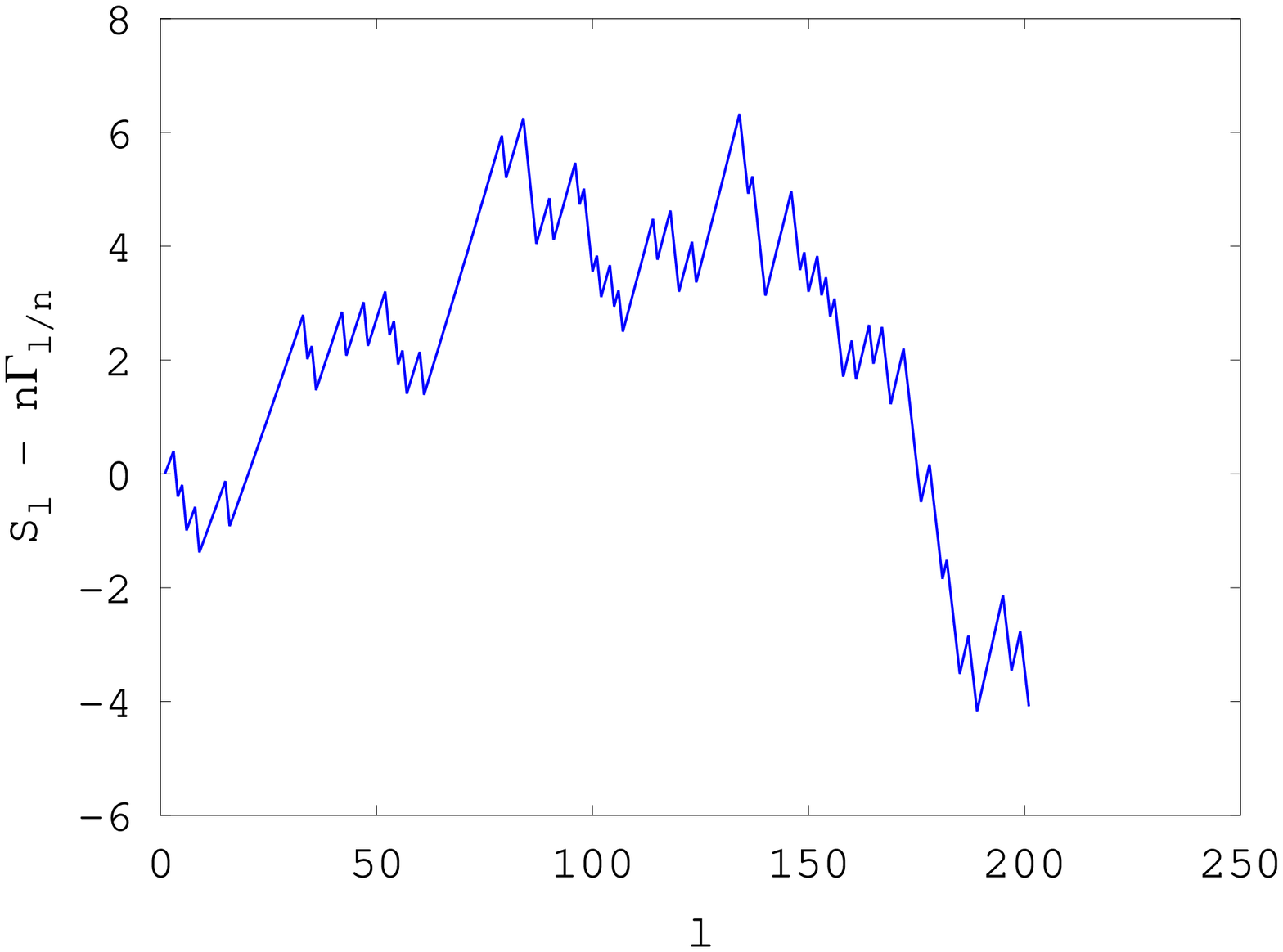}}}
\end{center}

\vspace{-0.8cm}
\caption{A sample path of $S_l$, $n\Gamma_{l/n}$ and their difference; $n=500$, $k_n = 200$ and $c=1$}
\label{f:samplepaths}
\end{figure}


\subsection{Comparison of push and pull for large networks}

\label{ss:compforlargen}
Theorem \ref{t:cthm} allows a simple comparison of the push and the pull algorithms
when $n$ is large.
In this subsection it will be easier to use a separate symbol to denote
the number of newly infected nodes in a pull round; let us use
$\tilde{Y}(n,k)$
for this purpose.
It is well known (see \cite{ozkasap2010analytical}) 
and simple to see that $\tilde{Y}(n,k)$ is Binomial$(1-p_n,n-k)$ with failure probability
$p_n = \binom{n-k-1}{c} / \binom{n-1}{c}$.
\begin{proposition}\label{t:comppullpushmean}
Let $\mu$ and $k_n$ be as in \eqref{e:defmu}. Then
${\mathbb E}[\tilde{Y}(n,k_n)] > {\mathbb E}[Y(n,k_n)]$ for $n$ large.
\end{proposition}
\begin{proof}
Theorem \ref{t:cthm} says that  the average proportion of newly infected
nodes after a push round converges to
$\Gamma_{1-\mu} = \mu(1-e^{-c(1-\mu)}) $.
The expected proportion of newly infected nodes in a pull round will be 
${\mathbb E}[\tilde{Y}(n,k_n)/n]= (1-p_n)(n-k_n)/n$, the mean of 
Binomial$(1-p_n,n-k_n)$ divided by $n$.
Since $c$
is fixed, $p_n \rightarrow \mu^c.$ This and $(n-k_n)/n \rightarrow \mu$ imply
$\lim_n {\mathbb E}[\tilde{Y}(n,k_n)/n] = \mu (1-\mu^c).$
The inequality
\begin{equation}\label{e:comparemeans}
  \mu (1-\mu^c) > \mu\left(1-e^{-c(1-\mu)}\right) 
\end{equation}
for all $\mu$ and $c$ implies the statement of the proposition.
\end{proof}
\begin{corollary}\label{c:compare}
For all $\mu$, if $c$ is taken large enough, a single round of pull or
push is enough to infect the whole network.
\end{corollary}
\begin{proof}
Both sides of
\eqref{e:comparemeans} converge to $\mu$, the initial proportion of susceptible nodes
as $c\rightarrow \infty$.
\end{proof}

Several comments on these results and possible research directions that they suggest 
are as follows.
While the inequality \eqref{e:comparemeans} holds, the difference
between the two sides
is at most $0.06$ for $c=1$ and decreases
as $c$ increases (simple calculus shows the truth of these statements). 
Thus, the performance of these rounds are on average similar,
which explains the near performance of the pull and the push 
in the numerical example given in Figure \ref{f:N50}.

There is an important caveat to Proposition \ref{t:comppullpushmean}, which we
would like explain with an example.
As with all values of $\mu$, 
for $\mu$ close to $1$, i.e., when initially most of the nodes are susceptible,
a pull round infects 
on average
more nodes than push, as indeed claimed by Proposition
\ref{t:comppullpushmean}.
{\em But} 
a push round takes merely  $(1-\mu)n$ random selections whereas the pull takes $\mu n$;
for large $\mu$, a push round is a very small operation whereas a pull
round involves almost the whole network.
Thus, for a fairer comparison we think that it would be a good idea to take into
account the sizes of these operations. Such a comparison
can be undertaken in future work.

An interesting
comparison is when $\mu = 0.5$ and the initial number of the infected nodes equal the number of the susceptible ones.
In this case push and pull will involve the same number of random selections.
For $c=1$, \eqref{e:comparemeans} implies that
the pull round infects on average fifty percent of the susceptible nodes
whereas the push approach infects around forty percent.
For increasing values of $c$ the difference quickly disappears and for $c\ge 15$
a single round of either algorithm is enough to infect the whole network.

The foregoing discussion suggests the following heuristic:
in a network with few infected nodes, initially set $c$ to a relatively high
value (say between $10$ and $15$, if possible)
use push until half the network is infected and then switch
to pull and gradually decrease $c$. For smaller values of $c$, it will
be more advantageous to switch to pull earlier. Obviously, to turn these
ideas into a full fledged algorithm requires more work 
including a specification of how the nodes detect the infection
level in the network to do the switch. The design of such an algorithm and its
analysis can also be the subject of future work.

\subsection{First time to hit $n\Gamma_{1-\mu}$}\label{ss:asymptoticstn}
We have seen in the previous subsection that
for $\mu = 1/2$ and $c\ge 15$ one expects a single round to be enough
to infect the whole network.
In such cases, the number of rounds before the proportion of infected nodes
hits a certain level
becomes trivial (i.e. just $1$), and ``the number of random selections'' before the same event
becomes more useful and interesting.
\eqref{e:approxY} implies that, under \eqref{e:defmu},  the 
ratio of the number of newly
infected nodes to the number of nodes in the whole network
at the end of the first round has expectation approximately $\Gamma_{1-\mu}$. The number of random selections
needed to hit this level corresponds to the following stopping time of $S$:
$\tau_n \doteq \inf\{ l : S_l \ge n \Gamma_{1-\mu} \}$.
The goal of this subsection is to derive
approximations to the distribution of $\tau_n$ 
using the diffusion approximation of
Theorem
\ref{t:cthm}. 
The results we obtain will also be useful in subsection \ref{ss:verycool}
in finding the limits of $\nu_\lambda^{n}$, the number of {\em rounds} needed
before the infection level of the network is $\lambda.$ 
We would like to point out that $\tau_n$ cannot be studied
if one represents the result of a push round as a single random variable,
the ensuing analysis requires the use of the random walk representation.

Choose $t_n$ so that it solves
\begin{equation}\label{e:constrainttn}
\Gamma_{1-\mu} - \Gamma_{t_n} = C \sqrt\frac{\var(X_{1-\mu})}{n},
\end{equation}
where $C > 0$ is a large constant.
\eqref{e:varX} and \eqref{e:quadraticvar} imply
\begin{equation}\label{e:deftn}
t_n \doteq -\frac{1}{c} \log \left(\frac{1}{\mu}\left( \mu - \Gamma_{1-\mu} + \frac{C}{\sqrt{n} }\sqrt{\var(X_{1-\mu})}\right)
\right).
\end{equation}
Taylor expanding $\log$ in the last display
 around $1-\Gamma_{1-\mu}/\mu$
gives 
\begin{equation}\label{e:approxtn}
t_n = 
1-\mu -\frac{C \sqrt{\var(X_{1-\mu} ) }}{c(\mu - \Gamma_{1-\mu}) \sqrt{n}}
+O(1/n).
\end{equation}

\begin{proposition}\label{p:tn}
\begin{equation}\label{e:boundonTn}
\p( \tau_n \le n t_n ) \le  \sqrt{ \frac{2}{\pi} } \int_{C}^\infty e^\frac{-x^2}{2}dx,
\end{equation}
for $n$ large enough.
\end{proposition}

We refer the reader to the appendix for the proof.
Proposition \ref{p:tn} implies that, once we choose $C$ large enough,
with very high probability $\tau_n > nt_n$ and, by \eqref{e:approxtn},
$nt_n$
is only $O(\sqrt{n})$ steps away from $n(1-\mu) = k_n.$  
We will use
this in the proof of the next theorem to focus our attention on a small
neighborhood around $k_n$.
\begin{theorem}\label{t:asymptoticstn}
$(\tau_n - k_n )/\sqrt{n}  \Rightarrow N(0,v)$
where
$v\doteq
 \var(X_{1-\mu})/ (c ( \mu - \Gamma_{1-\mu} ))^2.$
\end{theorem}
\begin{proof}
Fix a finite interval $(a,b)$; our goal is to show
\begin{equation}\label{e:thm2toshow}
\lim_{n} \p\left( \frac{\tau_n - k_n} {\sqrt{n}} \in (a,b) \right) =
\frac{1}{\sqrt{2\pi v}} \int_a^b e^{-x^2/2v} dx.
\end{equation}
Choose $C$ in \eqref{e:constrainttn} so that
\[
C > \frac{c (\mu - \Gamma_{1-\mu} )}{\sqrt{\var(X_{1-\mu})}}\max(|a|,|b|)
\]
and define
$
{\bf t}_n \doteq  k_n/n - t_n.
$
Partition the event $E\doteq  \{ (\tau_n - n(1-\mu) )/\sqrt{n} \in (a,b)  \}$
as
$
( E \cap \{ \tau_n <  nt_n \}  ) \cup ( E \cap \{ \tau_n \ge nt_n \}).
$
Proposition \ref{p:tn} implies that, by increasing $C$, if necessary,
the probability of the
first of these sets can be made arbitrarily small. Furthermore,
on the set $E$ the greatest value
that $\tau_n$ can take is
$k_n + b \sqrt{n}$; by the choice of $C$ this is bounded above by
$k_n + n {\bf t}_n.$ 
These imply that
we can replace $\tau_n = \inf \{l: S_l \ge n\Gamma_{1-\mu}\}$ 
in \eqref{e:thm2toshow} with
$\tau_n' \doteq \inf \{ l \in I_n: S_l \ge n\Gamma_{1-\mu}\}$,
where $I_n \doteq k_n + n {\bf t}_n( -1,1)$. 
Thus in the rest of this argument we will prove
\begin{equation}\label{e:Tn'scaled}
\lim_{n} \p\left( \frac{\tau_n' - k_n} {\sqrt{n}} \in (a,b) \right)
=
\frac{1}{\sqrt{2\pi v}} \int_a^b e^{-x^2/2v} dx.
\end{equation}
For this,
it is enough to study the asymptotics of the
dynamics of $S$ in the interval $I_n$. 
To do so, define the scaled process
\[
\hat{X}^n_{\hat{t}} \doteq \frac{1}{\sqrt{n}} \left( 
S_{\lfloor k_n + \hat{t}\sqrt{n}\rfloor} -n\Gamma_{1-\mu}\right).
\]
The scaled time $\hat{\tau}_n \doteq \frac{\tau_n' - k_n}{\sqrt{n}}$ of 
\eqref{e:Tn'scaled}
is the first time
the process $\hat{X}^n$ hits $0$. Thus to find its limit
distribution it is enough to compute the weak limit of $\hat{X}^{n}$,
which we will now do.

The time interval $J_n\doteq [-\theta_n,\theta_n] =  \sqrt{n}{\bf t}_n[-1,1]$ 
for the process
$\hat{X}^{n}$ corresponds exactly to the time interval $I_n$ for $S$
and therefore, we will be studying $\hat{X}^n$ on $J_n.$ 
Note that $t=0$ is the middle of $J_n$ and corresponds to time $k_n$ of $S$,
which is the last step of the first push round. The end of $J_n$ is
the time point $\theta_n = \sqrt{n} {\bf t}_n$. $k_n/n \rightarrow
(1-\mu)$ and \eqref{e:approxtn} imply 
\[
\theta_n 
\rightarrow \theta \doteq 
\frac{ C \sqrt{\var(X_{1-\mu})}}{c(\mu-\Gamma_{1-\mu})}.\]
$J_n = [-\theta_n,\theta_n]$ is symmetric around $0$ and its starting point $-\theta_n$
converges
to $-\theta.$ Then the limit process $\hat{X}$ 
of  $\hat{X}^{n}$ will be running
on the interval $[-\theta,\theta]$. 
Note that the initial point of $\hat{X}^n$ is $\hat{X}^{n}_{-\theta_n}=
\frac{1}{\sqrt{n}} (S_{t_n}-n\Gamma_{1-\mu})$. Theorem \ref{t:cthm},
$t_n/n \rightarrow (1-\mu)$, 
and \eqref{e:constrainttn} imply that this random variable converges
weakly to a normal random variable with mean $-C\sqrt{\var(X_{1-\mu})} < 0$
and variance $\var(X_{1-\mu})$. Hence, this is the distribution of the
limit
$\hat{X}(-\theta).$

To compute the dynamics of $\hat{X}$ one proceeds parallel to the proof
of Theorem \ref{t:cthm}.
Fix $(\hat{x},\hat{t}) \in {\mathbb R} \times [-\theta,\theta]$
and define
\[
\hat{T}_n \doteq 
{\mathbb E}_{(\hat{x},\hat{t})}
\left[f\left(\hat{X}^{n}_{\hat{t}+ 1/\sqrt{n}}\right)
\right],
\hat{A}_n  \doteq \sqrt{n} (\hat{T}_n - I),
\]
where $f$ is a smooth function on ${\mathbb R}$ with compact support.
The subscript $(\hat{x},\hat{t})$ 
of the expectation means that we are conditioning
of $\hat{X}^{n}_t = \hat{x}.$ It remains to compute 
$\hat{A}_n$. This computation is parallel to the
arguments given in the proof of Theorem \ref{t:cthm} with one important
difference: now time is scaled by $1/\sqrt{n}$ rather than
$1/n$. Thus, we omit the details and directly write down the limit:
$\lim_{n\rightarrow \infty} A_n f = f'(\hat{x}) c(\mu - \Gamma_{1-\mu} ).$
The right side of the last display is the generator
of the process
\begin{equation}\label{e:dynamicsXhat}
 \hat{X}_{\hat{t}} \doteq \hat{X}_{-\theta} + c(\mu - \Gamma_{1-\mu})(\hat{t} +\theta)
\end{equation}
whose randomness is completely determined by its initial position $\hat{X}_{-\theta}$.
Exactly the same line of arguments as in the proof of Theorem \ref{t:cthm}
now
imply $\hat{X}^{n} \Rightarrow \hat{X}.$
We are interested in the limit of $\p( \hat{\tau}_n \in (a,b) )$,
the probability that $\hat{X}^{n}$ hits $0$ between time points $a$
and $b$. The weak limit we have just 
established implies that this limit equals $\p( \hat{\tau} \in (a,b))$
where $\hat{\tau}$ is the first time when the limit process $\hat{X}$ hits
$0$. \eqref{e:dynamicsXhat} implies that it will take $\hat{X}$
$
\frac{-\hat{X}_{-\theta}}{c(\mu - \Gamma_{1-\mu})}
$
unit of time to hit $0$. We subtract from this $\theta$ to convert it
to the time unit of the limit time interval $[-\theta,\theta]$:
$\tau' = 
-\hat{X}_{-\theta} / c(\mu - \Gamma_{1-\mu}) - \theta.$
But this is a random variable with mean $0$ and variance $v$. Hence we have
\eqref{e:Tn'scaled}.

\end{proof}
Here is a numerical example for Theorem \ref{t:asymptoticstn}.
Suppose we have a network with $n=5000$ nodes of which $k=200$ are initially 
infected, i.e., $\mu = 0.96$. Suppose that fanout is $c=5$.
For this network, the $v$ that appears in Theorem \ref{t:asymptoticstn} is 
$v \approx 0.0012$.
We know from Theorem \ref{t:cthm} that on average at the end
of the first push round the total number of infected nodes will be
$ 200 + n \Gamma_{1-\mu} \approx 1070$. Now Theorem \ref{t:asymptoticstn}
says that the first push round will attain this 
infection level with probability approximately equal to $1/2$:
and if it does, 
this will almost certainly happen in the last $4\sqrt{5000v } \approx 10 $
steps of the push round (which lasts a total of $k=200$ steps); this
network hitting $1070$ infected nodes before the last $10$ random selections 
is as likely as a normal random variable
being $4$ standard deviations below its mean.
The same theorem implies that with probability $1/2$ 
this level will be attained approximately
in the first $10$ steps of the next round. 

For $ 0 < \lambda < \mu$, define
$\tau^{n}_\lambda \doteq \inf \{l: S_l \ge n \lambda \}.$
$\tau^{n}_\lambda$ is the number of random selections 
before the proportion of newly 
infected nodes has reached $\lambda$; in terms of the random walk $S$,
it is the first time $S$ reaches the level $n\lambda.$

Define
\begin{equation}\label{e:bartaudef}
\bar{\tau}_\lambda(\mu) \doteq 
 -\frac{1}{c} \log\left( 1-  \frac{\lambda}{\mu}\right)
\end{equation}
$\bar{\tau}$, as a function of $\lambda$, is the inverse function
of the fluid limit $\Gamma_t(\mu) = \mu ( 1 -e^{-ct}) $ with respect to $t$.
Dynamically, $\bar{\tau}_\lambda$ is the first time $\Gamma$ hits the level
$\lambda.$ Because $\Gamma$ is deterministic, so is this hitting time.
A quick examination of the proof of Theorem \ref{t:asymptoticstn} reveals
that if one replaces $1-\mu$ with $\bar{\tau}_\lambda$ 
and $\Gamma_{1-\mu}$ with $\lambda$,
the proof continues to work exactly as is.
This gives us the following generalization of
Theorem \ref{t:asymptoticstn}:
\begin{theorem}\label{t:asympalpha}
Let
$\tau^{n}_\lambda$ and $\bar{\tau}_\lambda$ be as above. 
$(\tau^{n}_\lambda - n\bar\tau_\lambda )/\sqrt{n}  \Rightarrow N(0,v) $
with
$v = \var(X_{\bar{\tau}_\lambda})/ (c(\mu - \lambda))^2.$
\end{theorem}

\subsection{Fluid limit of the whole push algorithm}\label{ss:fluidlimit}
As suggested in the beginning of this section,
instead of the number of newly infected nodes, one can keep
track of their proportion in the network. 
This amounts to dividing $S$ by $n$. 
Theorem \ref{t:cthm} gives us the following approximation for 
the proportion process:
$S_l/n \approx \Gamma_{l/n} + \frac{1}{\sqrt{n}} X_{l/n}$
Thus as $n$ goes to infinity, the proportion process $S/n$ converges
to its fluid limit $\Gamma$. 
In subsection \ref{ss:Smodel} we have
argued that the random walk $S$, when ran indefinitely, is a model
for the whole push algorithm. This implies that $t\rightarrow \Gamma_t$
for $t\in(0,\infty)$ is a representation of
the fluid limit of the whole push algorithm.
To break it into rounds, we proceed as follows. Define
$\varphi_0 \doteq (1-\mu)$, this is the initial proportion of infected nodes
in the fluid limit. The first round in the prelimit lasts $k_n$ steps;
in the scaled limit time, this corresponds to the time point 
$\lim_n k_n/n = 1-\mu$, thus
the first round of the fluid limit ends at time $1-\mu$ and
the total proportion
of infected nodes after the first round is
$\varphi_1 \doteq \varphi_0 + \Gamma_{\varphi_0}$.
The next fluid round
will last a time interval of length $\varphi_1$ and will add an additional
$\Gamma_{\varphi_0 + \varphi_1} - \Gamma_{\varphi_0}$
proportion of infected nodes, bringing
the total proportion of infected nodes at the end of the second round to
$\varphi_2 \doteq \varphi_0 + \Gamma_{\varphi_1+\varphi_0}$
In general, the proportion of infected
nodes at the end of the $i^{th}$ round of the fluid limit will be
\begin{equation}\label{e:sdef}
\varphi_i = \varphi_{0} +\Gamma_{c_{i-1}},
\end{equation}
with $c_{i-1} \doteq \sum_{j=0}^{i-1} \varphi_j$.
The construction of $\{\varphi_0,\varphi_1,\varphi_2,...\}$ 
is the fluid limit version of that
of the sequence $(I_0,I_1,I_2,...)$ given in subsection \ref{ss:Smodel};
and indeed, Theorem \ref{t:cthm} implies 
$I_i/n \Rightarrow \varphi_i.$
The sequence $\{\varphi_i\}$ is increasing and deterministic;
in the fluid limit, each round infects a deterministic proportion of the
network. 
In the next subsection we will use the sequence
$\{\varphi_i\}$ to compute the weak limit of 
$\nu^{n}_\lambda$, the
number of push rounds
in a network
with $n$ nodes 
before the proportion of the infected nodes in the network hits $\lambda$.

\subsection{Weak limit of $\nu_\lambda^{n}$}\label{ss:verycool}
Recall that $\nu_\lambda^{n}$,  defined in
\eqref{e:nunlambda}, 
is the number of push rounds
needed before the proportion of the infected nodes in the network reaches 
$\lambda \in (0,1)$. Section \ref{s:numerical} presents  numerical
computations of the expectation of this random variable for a network
with $500$ nodes and for different fanout values and makes
several observations
about the results. Here,
we derive the weak limit of this random variable using
the sequence $\{\varphi_i \}$ of \eqref{e:sdef} and
Theorems \ref{t:cthm} and \ref{t:asympalpha}.

Define
$\bar{\nu}_\lambda\doteq \inf\{i: \varphi_i \ge \lambda \};$
$\bar{\nu}$ is the number of rounds that the fluid limit network needs
so that its proportion of infected nodes equals $\lambda$. Because
$\varphi$ is an increasing deterministic sequence, $\bar{\nu}_\lambda$ can be 
characterized as follows:
$\bar{\nu}_\lambda = i$ 
for $\lambda \in [\varphi_{i}, \varphi_{i+1})$.
Hence, as function of $\lambda$, $\bar{\nu}$ is
right continuous, piecewise constant and it jumps precisely by $1$
at the points $\{\varphi_1,\varphi_2,\varphi_3,...\}.$
As the final step of our analysis of the push algorithm we prove
\begin{theorem}\label{t:nulambdalimit}
$ \nu_\lambda^{n} \Rightarrow \bar{\nu}_{\lambda} $
for $\lambda \in (0,1) - \{\varphi_0,\varphi_1,\varphi_2,...\}$.
\end{theorem}
\begin{proof}
$ \nu_\lambda = \bar{\nu}_\lambda =0$ for $\lambda \le \varphi_0=1-\mu$; i.e., if
$\lambda$ is less than $1-\mu$ (the initial proportion of infected nodes),
the network already has more than $\lambda$ proportion of infected nodes
before any round begins.

To keep the 
proof short, we will treat the first two rounds; 
an argument that covers all rounds
will involve the same ideas.
Fix a $\lambda \in (\varphi_0,\varphi_1)$; we would like to show that
$\nu_\alpha^{n} \Rightarrow \bar{\nu}_\lambda = 1$, i.e., 
as $n$ goes to $\infty$, 
the infection level $\lambda$ is attained in the first round with
probability approaching $1$.
$\lambda < \varphi_1$, and the definitions \eqref{e:sdef} and \eqref{e:bartaudef}
of $\varphi_1$ and $\bar{\tau}_\lambda$ imply
\begin{equation}\label{e:lambdaless}
\bar{\tau}_\lambda < 1-\mu;
\end{equation}
this last inequality can also be expressed as follows:
at time $1-\mu$ the proportion of infected nodes in the fluid limit is
$\varphi_1$; $\lambda$ being strictly less than $\varphi_1$ and the
fluid limit $\Gamma$ being strictly increasing
and deterministic, it must be that the first time the fluid limit
has reached the infection level $\lambda$ must be before time $1-\mu$. 

$\nu_\alpha^{n} =1$ 
if and only if
$\tau^{n}_\lambda \le k_n$,
which is the same inequality as
\begin{equation}\label{e:taucenterscale}
\frac{1}{\sqrt{n}} \left( \frac{1}{\sqrt{n}}\left(
\tau^{n}_\lambda - n \bar{\tau}_\lambda\right) \right)\le 
 k_n/n - \bar{\tau}_\lambda.
\end{equation}
The first term on the left is a constant and converges to $0$;
Theorem \ref{t:asympalpha} says that the second term on the left
converges weakly to a finite random variable. Thus, their product
converges weakly to $0$. 
 \eqref{e:lambdaless} and $k_n/n \rightarrow (1-\mu)$ imply, on the other hand,
that the limit of the right side is strictly greater than $0$.
Thus, the probability
of the event expressed in this display, that is, the probability that
$\nu_\lambda = 1$, indeed converges to $1$. 

For $\lambda \in (\varphi_1,\varphi_2)$, we would like to show $\nu_\lambda^{n}$
converges weakly to $\bar{\nu}_\lambda = 2$, 
i.e., we want to show that the probability
of the event
$k_n  <  \tau_\lambda < k_n + S_{k_n}$
converges to $1$ ($k_n$ is the total number of random selections in the
first round and $k_n + S_{k_n}$ is the same total after the second round). 
A rescaling and centering similar to \eqref{e:taucenterscale}
and Theorems \ref{t:cthm}, \ref{t:asympalpha} imply this.
\end{proof}
An interesting question is the weak limit of 
$\nu_{\varphi_i}^{n}$. We have already covered the case $i=1$ in the
argument given in the numerical example
following Theorem \ref{t:asymptoticstn}: $\nu_{\varphi_1}^{n} \Rightarrow T_1$
where $T_1 \in \{1,2\}$ and takes these value with equal probability. 
The weak limit of the whole sequence
$\left\{\nu_{\varphi_i}^{n}, i=1,2,3,... \right\}$ requires a longer analysis and we leave it to future work.

Figure \ref{f:nuandN} shows two graphs:  first is that of
$\bar{\nu}_\lambda$ for an initial infection rate of $\mu =0.01$ and fanout
$c=3$; the second is $N(\lambda,5= 500\mu) ={\mathbb E}_5[\nu^{(500)}_\lambda]$;
the graph of $N(\lambda, 50)$ for $c\in \{1,7\}$ 
has been given earlier in Figure \ref{f:N50}.
\ninseps{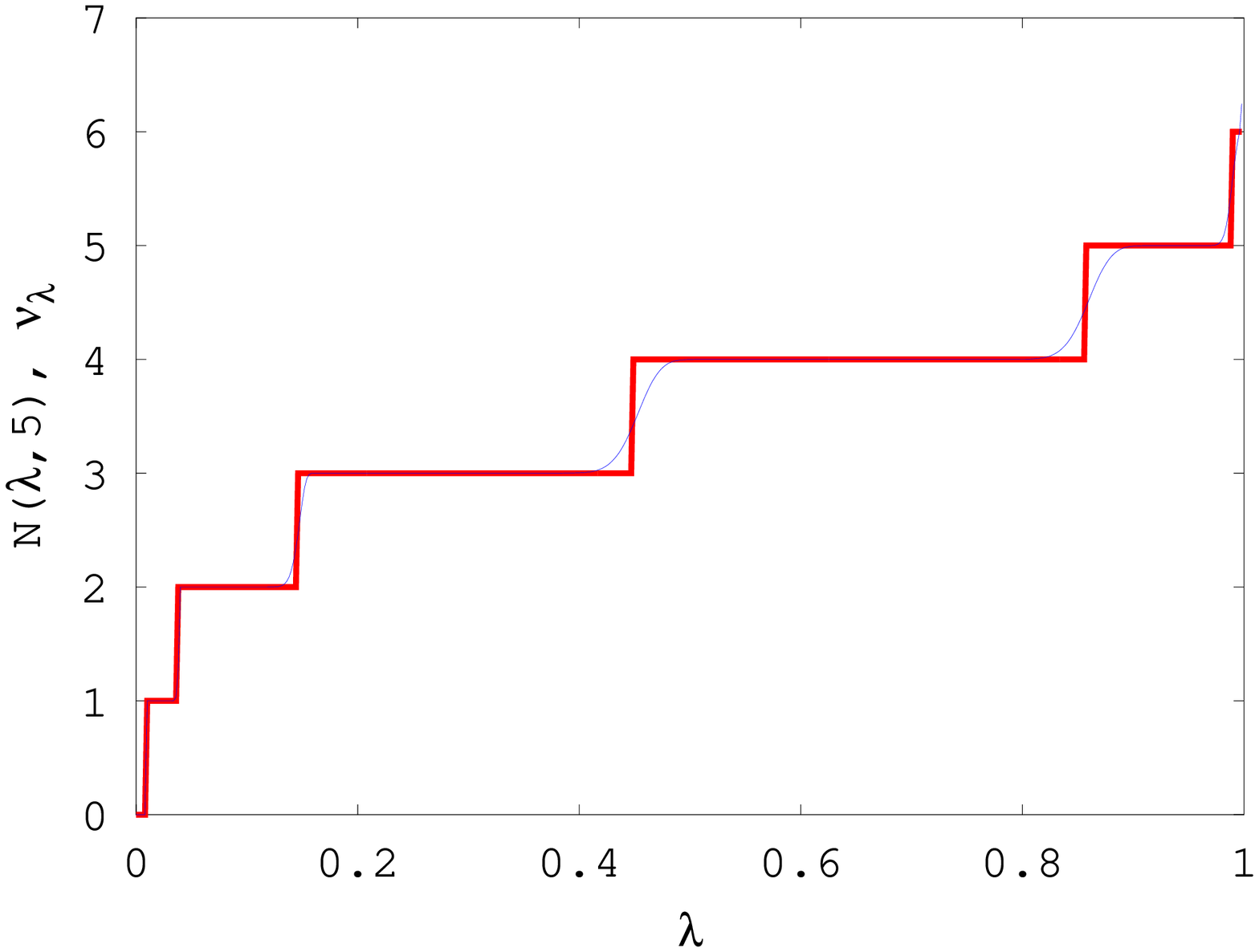}{The graphs of $\bar{\nu}_\lambda$ (the thicker curve)
and $N(\lambda,5)$; both for $c=3$}
{0.45}
Clearly, as Theorem \ref{t:nulambdalimit} implies, 
these graphs overlap except
around $\{\varphi_i\}$; at $\varphi_i$,
$\bar{\nu}$ jumps by $1$ while $N$ rapidly goes from $i-1$ to $i$. Theorem \ref{t:nulambdalimit}
also implies that the $\lambda$ interval over which this transition occurs shrinks to the single
point $\varphi_i$ as the network size grows to $\infty.$

\section{Conclusion}\label{s:conclusion}
The starting point
of the analysis of the present paper is taking
the random selections in a push round as the atomic operation of the
push algorithm and defining
a random walk whose each step corresponds to a random selection. The
main body of the analysis consists of computing weak limits of this
random walk and its various functions.
We expect
these ideas
to be directly applicable
to other epidemic algorithms (such as
those considered in \cite{karp2000randomized, sanghavi2007gossiping})
over
fully connected networks. 

Fully connectedness is one of the natural limits of 
the collection of possible topologies over a given collection of nodes.
As the topology of a network approaches full connectivity
one expects the results for the latter to be good approximations for 
the former.
One can further use results on quantities in fully connected networks as upper and
lower bounds on the same quantities in
other network topologies. After a detailed analysis of this basic case,
an interesting and important direction is the
generalization to more complex network topologies and
structures, such as those considered
in \cite{chierichetti2011rumor,ganesh2005effect, anagnostopoulos2011information}.

To the best of our understanding, most of the literature on the
asymptotic analysis of epidemic algorithms focus on obtaining upperbounds
on the tail probabilities of the total
dissemination time, $\nu_1^{n}$ in our notation.
For this, authors often use
large deviations results such as Chernoff's bound. However, the
underlying processes in these models may have fluid and diffusion limits
and these limits can give more precise information about the
distribution of key random variables, such as the total dissemination time.
We hope that the present work provides an example of how
this path can be followed in the context of a simple model.

\section*{Acknowledgement}
Ali Devin Sezer's work on this article has been supported by the 
Rbuce-up European Marie Curie project, \url{http://www.rbuce-up.eu/}.
\appendix
\section{Proofs}
\begin{proof}[Proof of Theorem \ref{t:cthm}]
To avoid confusion between discrete and continuous time parameters, we will
show continuous time parameters in parentheses, e.g., we will write $X^n(t)$
instead of $X^n_t.$
We begin by assuming that  $c=1$, the modifications for $c>1$ will be straightforward.
$k_n$ is assumed to grow with $n$ so that $\lim_n (n-k_n)/n = \mu \in(0,1)$.
Let $C_b({\mathbb R}^2)$ denote the set of bounded and continuous
functions on ${\mathbb R}^2$; and $C^2_0({\mathbb R}^2)$ the set
of twice differentiable functions with compact support with
continuous  Hessians. Define $T_n:C_b({\mathbb R}^2) \rightarrow
C_b({\mathbb R}^2)$ as 
\begin{equation}\label{e:defTn}
[T_n(f)](x) = {\mathbb E}_x [ f(X^n(1/n), \Gamma^n(1/n))], x\in
{\mathbb R}^2,
\end{equation}
where the subscript $x$ of ${\mathbb E}_x$ denotes conditioning on
$(X^n(0),\Gamma^n(0)) = x = (x_1,x_2).$

Define
\begin{equation}\label{e:defAn}
A_n \doteq n (T_n - I),
\end{equation}
where $I$ denotes the identity operator on $C_b({\mathbb R}^2).$
Let us compute $A_n$ explicitly for $ f\in C^2_0({\mathbb R}^2)$.
The expectation in \eqref{e:defTn} is conditioned on\\
$(X^n(0),\Gamma^n(0)) = (x_1,x_2)$, i.e., on
\begin{equation}\label{e:condition}
\frac{S_0 - \gamma_0}{\sqrt{n}} = x_1, \frac{\gamma_0}{n} = x_2.
\end{equation}
This, \eqref{e:ilk}, \eqref{e:ilkrec} 
and the definition of
$X^n$ imply
\begin{align*}
X^n(1/n) = \frac{ S_1 - \gamma_1}{\sqrt{n}}&= \frac{S_0
-\gamma_0}{\sqrt{n}}
+ \frac{S_1 - S_0}{\sqrt{n}} - \frac{\gamma_1 - \gamma_0}{\sqrt{n}}\\
&= x_1 +\left( X_1 -\left( \frac{n-k_n}{n-1} - \frac{n}{n-1}x_2\right)\right)/\sqrt{n}\\
\Gamma^n(1/n) &= \frac{\gamma_1}{n} = \frac{1}{n}\left( \gamma_0 \frac{n-2}{n-1} + \frac{n-k_n}{n-1}\right)\\
&=  x_2 -  \frac{x_2}{n-1} + \frac{n-k_n}{n(n-1)}.
\end{align*}
Define
\begin{align} \label{star}
Y_1 & \doteq  X_1 - \left( \frac{n-k_n}{n-1} - \frac{n}{n-1}x_2\right), \\  
y_2 & \doteq  -  \frac{x_2}{n-1} + \frac{n-k_n}{n(n-1)},~~~
\Delta \doteq ( Y_1/\sqrt{n}, y_2) \notag
\end{align}
where $X_1$ refers to the random variable given in
(\ref{e:bern}).
${\mathbb E}_x[Y_1]$ and $E_x[Y_1^2]$ will be useful in what
follows, so let us compute them first. Equations \eqref{e:condition}
and (\ref{e:bern}) give
\begin{equation}\label{e:expX1}
{\mathbb E}_x[X_1] = \frac{n-k_n}{n-1} - \frac{n}{n-1}x_2
-\frac{\sqrt{n}}{n-1}x_1,
\end{equation}
which gives
\begin{equation}\label{e:expY1}
{\mathbb E}_x[Y_1] = {\mathbb E}_x[X_1] - \left( \frac{n-k_n}{n-1} -
\frac{n}{n-1}x_2\right) = -\frac{\sqrt{n}}{n-1}x_1.
\end{equation}

On the other hand
${\mathbb E}_x[X_1]={\mathbb E}_x[X_1^2] = (n-k_n - nx_2)/(n-1)$ implies
\begin{align}\label{e:secY1}
{\mathbb E}_x[ Y_1^2] &= 
{\mathbb E}_x[ X_1^2]  + 2{\mathbb E}_x[Y_1]d_n -d_n^2\\
		& = (n-k_n - nx_2)/(n-1)  -\frac{\sqrt{n}}{n-1}x_1d_n -d_n^2\notag
\end{align}
where $d_n$ refers to the constant $(- (n-k_n)+ n x_2)/(n-1)$ in $Y_1$'s definition.

The expectation that occurs in the definition \eqref{e:defTn} of
$T_n$ written in terms of the vector $\Delta$ is
$
{\mathbb E}_x[ f(X^n(1/n), \Gamma^n(1/n))] = {\mathbb E}_x[ f( x +
\Delta )].
$
Let $Df$ denote the gradient of $c$ and $Hf$ its Hessian; and let
$\Delta^2$ denote the tensor product $\Delta \otimes \Delta.$ Using
$f$'s Taylor's expansion in the last display gives
\begin{equation}\label{e:taylorexpansion}
{\mathbb E}_x[ f(X^n(1/n),\Gamma^n(1/n))] = {\mathbb E}_x\left[
f(x) + \langle Df(x), \Delta\rangle  + \frac{1}{2} \langle
Hf(x+\theta_n), \Delta^2\rangle \right]
\end{equation}
where $\theta_n$ is a random vector that lies on the line segment
connecting $0$ to $\Delta.$ Let us deal with each of the terms that
appear in \eqref{e:taylorexpansion} one by one. The function $f(x)$
is not random and so it comes out of the expectation. The first
order term is
$\langle Df(x), \Delta \rangle = f_{x_1}(x) Y_1/\sqrt{n} + f_{x_2}(x) \, y_2.$
That the second term is deterministic and \eqref{e:expY1} give
\begin{equation}{\mathbb E}_x \left[ \langle Df(x), \Delta \rangle\right]
 =  -f_{x_1}(x) \frac{x_1}{n-1} + f_{x_2}(x) \, y_2.  \label{expinner}
\end{equation}
${\mathbb E}_x\left[ \langle Hf(\theta_n), \Delta^2\rangle\right]$
is
\begin{align*}
&
\frac{1}{n} {\mathbb E}_x\left[ f_{x_1,x_1}(x+\theta_n) Y_1^2
\right] + 2 y_2 \frac{1}{\sqrt{n}} {\mathbb E}_x
\left[ f_{x_1,x_2}(x+\theta_n) Y_1 \right]
\\
&~~~ 
+ (y_2)^2 f_{x_2,x_2}(x+\theta_n) .
\end{align*}
Substituting these in \eqref{e:defAn} yields
\begin{align}\label{e:AnafterTaylor}
A_n f&= -f_{x_1} \frac{n}{n-1}x_1 +  f_{x_2}(x) (n y_2 ) +
\frac{1}{2}
{\mathbb E}_x\left[ f_{x_1,x_1}(x+\theta_n) Y_1^2 \right] \\
&~~~ +y_2\sqrt{n} {\mathbb E}_x \notag \left[
f_{x_1,x_2}(x+\theta_n) Y_1 \right] +\frac{1}{2} n (y_2)^2
f_{x_2,x_2}(x+\theta_n)
\end{align}
Now let us compute the limits of each of the terms in the last sum
 as $n\rightarrow \infty$. The first term converges to $-f_{x_1} x_1$.
The definition of $y_2$ and \eqref{e:defmu} imply
\begin{equation}\label{e:firstorderterm}
\lim_{n \rightarrow \infty} f_{x_2}(x) (n y_2 )  = f_{x_2}(
\mu-x_2).
\end{equation}
Note that
\begin{equation}\label{e:theta0}
\theta_n \rightarrow 0
\end{equation}
uniformly, because $\Delta$ is bounded. This and the continuity of
$Hf$ imply
\begin{equation}\label{e:x11term}
\lim_{n\rightarrow \infty} \frac{1}{2} {\mathbb E}_x\left[
f_{x_1,x_1}(x+\theta_n) Y_1^2 \right]  = \frac{1}{2}
f_{x_1,x_1}(x) \lim_{n\rightarrow \infty} {\mathbb E}_x\left[
Y_1^2 \right].
\end{equation}
The observation \eqref{e:theta0}, continuity of $Hf$ and $|(y_2)^2| = O(1/n^2)$
imply
\begin{equation}\label{e:x22term}
\lim_{n\rightarrow 0} \frac{1}{2} n (y_2)^2
f_{x_2,x_2}(x+\theta_n)
 = 0.
\end{equation}
Boundedness of $Y_1$, \eqref{e:expY1}, \eqref{e:theta0},
 continuity of $Hf$ an $|y_2| = O(1/n)$ imply
\begin{equation}\label{e:mixedterm}
\lim_{n\rightarrow \infty} y_2\sqrt{n} {\mathbb E}_x \left[
f_{x_1,x_2}(x+\theta_n) Y_1 \right] = 0.
\end{equation}
Letting $n\rightarrow \infty$ in \eqref{e:secY1} gives
\begin{equation}
\label{e:varlimitX1}
\lim_{n\rightarrow \infty} {\mathbb E}_x\left[ Y_1^2 \right] =
(\mu - x_2) ( 1 - (\mu-x_2)).
\end{equation}

The equations \eqref{e:AnafterTaylor}, along with
\eqref{e:firstorderterm},
 \eqref{e:x11term}, \eqref{e:x22term}, \eqref{e:mixedterm} and
\eqref{e:varlimitX1} yield
$
\lim_{n\rightarrow \infty} A_n f = Af
$
where
$
Af \doteq \langle Df, (-x_1, \mu -x_2 ) \rangle + \frac{1}{2}
(\mu-x_2)(1-(\mu-x_2)) f_{x_1,x_1}.
$
One can check directly that $A$ is the infinitesimal generator of
the semigroup $T(\cdot)$
 defined by the process $(X_t,\Gamma_t)$.
It follows from its definition that $T(\cdot)$ is a Feller
semigroup on $C_b({\mathbb R}^2)$. Furthermore,
\cite[Proposition 3.2,
page 17]{ethier1986t} 
imply that
$C^\infty({\mathbb R}^2)$ forms a core for the generator $A$. Thus,
\cite[Theorem 1.2, page 31]{ethier1986t} and \cite[Theorem
2.6, page 168]{ethier1986t} imply $(X^n,\Gamma^n)\Rightarrow
(X,\Gamma)$.

Modifications for $c>1$ are as follows. 
The variables $Y_1$ and $y_2$
are now defined as
 \[
 Y_1 \doteq X_1 - c\, \left( \frac{n-k_n}{n-1} - \frac{n}{n-1}x_2\right),
\qquad y_2 \doteq -  \frac{x_2\,c }{n-1} + \frac{(n-k_n)\, c}{n(n-1)};
 \]
 where $X_1$ has the hypergeometric distribution \eqref{X1}.
Then
\begin{equation}\label{e:Ey1c}
\E_x[Y_1]=-c\,\frac{\sqrt{n}}{n-1} x_1.
\end{equation}
\eqref{e:secY1} becomes
${\mathbb E}_x[ Y_1^2] = {\mathbb E}_x[ X_1^2]  + 2{\mathbb E}_x[Y_1]d_n -d_n^2 
=  {\mathbb E}_x[ X_1^2]  -\frac{c\sqrt{n}}{n-1}x_1d_n -d_n^2$
where $d_n$ is now  $c(- (n-k_n)+ n x_2)/(n-1)$. For the asymptotic analysis,
we only need the limit of the last display as $n\rightarrow \infty.$ For $n$
large, one can approximate the hypergeometric $X_1$ as Binomial$(p^*,c)$ 
with success probability $p^* = (k_n - nx_2)/n$, whose second moment is
$ c p^* + c(c-1)(p^*)^2.$ Substituting this in the last display and letting
$n\rightarrow \infty$ we get
\begin{equation}\label{e:varlimitx1c}
\lim_{n\rightarrow \infty} \E_x [Y_1^2] = c\, (\mu-x_2)(1-(\mu-x_2))
\end{equation}
which generalizes \eqref{e:varlimitX1} to $ c > 1.$
On the other hand, we have\\
${\mathbb E}_x \left[ \langle Df(x), \Delta \rangle\right]
 =  -f_{x_1}(x) \, c\,\frac{x_1}{n-1} + f_{x_2}(x) \, y_2$
(generalization of \eqref{expinner}).
It follows that
\begin{align*}
A_n f&= -f_{x_1} \,c\, \frac{n}{n-1}x_1 +  f_{x_2}(x) (n y_2 ) +
\frac{1}{2}
{\mathbb E}_x\left[ f_{x_1,x_1}(x+\theta_n) Y_1^2 \right] \\
&~~~ +y_2\sqrt{n} {\mathbb E}_x \notag \left[
f_{x_1,x_2}(x+\theta_n) Y_1 \right] +\frac{1}{2} n (y_2)^2
f_{x_2,x_2}(x+\theta_n).
\end{align*}
Same arguments as in the $c=1$ case give
$\lim_{n \rightarrow \infty} f_{x_2}(x) (n y_2 )  = f_{x_2}\, c\, (
\mu-x_2).$
The last two displays \eqref{e:Ey1c} and \eqref{e:varlimitx1c} imply
$
\lim_{n\rightarrow \infty} A_n f = Af
$
where
$
Af \doteq \langle Df, c\, (-x_1, \mu -x_2 ) \rangle + \frac{1}{2} \,
c\,  (\mu-x_2)(1-(\mu-x_2)) f_{x_1,x_1}.
$
The rest of the proof is the same as in the case of $c=1.$
\end{proof}
\begin{proof}[Proof of Proposition \ref{p:tn}]
Define the stopping time 
\[
\tau_n' \doteq \inf \left 
\{ l : \frac{1}{\sqrt{n} } (S_l - \gamma_l) \ge \frac{ n\Gamma_{1-\mu} - \gamma_{\lfloor nt_n\rfloor } }{\sqrt{n}} \right\}.
\]
$S_l \ge n \Gamma_{1-\mu}$
is the same as
$ (S_l - \gamma_l)/\sqrt{n} \ge 
(n\Gamma_{1-\mu} - \gamma_l)/\sqrt{n}$.
That $\gamma_l$ is increasing in $l$ implies that the right side of this inequality is decreasing
in $l$. This implies 
\begin{equation}\label{e:boundTnTn'}
\{ \tau_n \le n t_n \} \subset \{ \tau_n' \le n t_n \}.
\end{equation}
We know from Theorem \ref{t:cthm} 
that $\gamma_{\lfloor n t_n \rfloor}/n \rightarrow \Gamma_t.$ 
This and \eqref{e:constrainttn}
imply
$\lim_n (n\Gamma_{1-\mu} - \gamma_{\lfloor nt_n\rfloor } )/{\sqrt{n}} $
$\rightarrow$ $ C \sqrt{\var(X_{1-\mu})}$.
Now define 
$\tau' \doteq \inf $ $ \{t: \bar{X}_t \ge C \sqrt{\var(\bar{X}_{1-\mu})} \}.$
The last two displays, that $t\rightarrow e^{ct}$ is monotone increasing,
 $t_n\rightarrow (1-\mu)$ (see \eqref{e:approxtn})
 and Theorem \ref{t:cthm} imply 
$\lim_{n\rightarrow \infty} \p(\{ \tau_n' \le n t_n \})$ $\le$
$ \p (\tau' \le (1-\mu) ).$
$\bar{X}$ is a stochastic integral against a Brownian motion. Therefore,
if we measure time using its quadratic variation it will be a standard
Brownian motion \cite[Theorem 4.6, page 174]{karatzas1991brownian}. This and
 \cite[Equation (6.3), page 80]{karatzas1991brownian}
give
$\p( \tau' \le (1-\mu) ) = \sqrt{ \frac{2}{\pi} } \int_{C}^{\infty} e^\frac{-x^2}{2}dx.$
The last equality  and \eqref{e:boundTnTn'} imply \eqref{e:boundonTn}.
\end{proof}
\bibliography{pushanalysis}

\end{document}